\documentclass[a4paper,11pt,leqno]{amsart}
\usepackage{amsfonts,amssymb,amsmath,amsthm}
\usepackage{url}
\usepackage{enumerate}

\usepackage{graphicx}
\usepackage{latexsym,bm}
\usepackage{mathrsfs,secdot}
\usepackage{amscd}

\DeclareFontEncoding{OT2}{}{}
\DeclareFontSubstitution{OT2}{cmr}{m}{n}
\DeclareSymbolFont{cyss}{OT2}{wncyss}{m}{n}
\DeclareMathSymbol{\sh}{\mathbin}{cyss}{`x}

\newtheorem{thrm}{Theorem}[section]
\newtheorem{lem}[thrm]{Lemma}
\newtheorem{prop}[thrm]{Proposition}
\newtheorem{cor}[thrm]{Corollary}
\theoremstyle{definition}
\newtheorem{definition}[thrm]{Definition}
\newtheorem{remark}[thrm]{Remark}
\newtheorem{example}[thrm]{Example}
\numberwithin{equation}{section}

\def\fs{\Psi}

\def\abs#1{\lvert#1\rvert}

\def\q{{\mathbb{Q}}}

\def\w{\omega}
\def\cc{{\mathcal{C}}}
\def\ee{{\varepsilon}}

\DeclareRobustCommand\openone{\leavevmode\hbox{\small1\normalsize\kern-.33em1}}

\allowdisplaybreaks

\author{Yasushi Komori}
\address{Y. Komori, Department of Mathematics, Rikkyo University, Nishi-Ikebukuro, Toshima-ku, Tokyo 171-8501, Japan}
\email{komori@rikkyo.ac.jp}
\thanks{}

\author{Kohji Matsumoto}
\address{K. Matsumoto, Graduate School of Mathematics, Nagoya University, Chikusa-ku, Nagoya 464-8602 Japan}
\email{kohjimat@math.nagoya-u.ac.jp}
\thanks{}

\author{Hirofumi Tsumura}
\address{H. Tsumura, Department of Mathematics and Information Sciences, Tokyo Metropolitan University, 1-1, Minami-Ohsawa, Hachioji, Tokyo 192-0397 Japan}
\email{tsumura@tmu.ac.jp}
\thanks{}

\keywords{Multiple Zeta values, Witten zeta-functions, root systems, Riemann's zeta-function}
\subjclass[2010]{Primary 11M32, Secondly 11M06}
\begin{document}

\title[A study on multiple zeta values]{A study on multiple zeta values from the viewpoint of zeta-functions of root systems}
\date{}

\begin{abstract}
We study multiple zeta values (MZVs) from the viewpoint of zeta-functions associated with the root systems which we have studied in our previous papers. In fact, the $r$-ple zeta-functions of Euler-Zagier type can be regarded as the zeta-function associated with a certain sub-root system of type $C_r$. Hence, by the action of the Weyl group, we can find new aspects of MZVs which imply that the well-known 
formula for MZVs given by Hoffman and Zagier coincides with Witten's volume formula 
associated with the above sub-root system of type $C_r$. Also, from this observation, 
we can prove some new formulas which especially include the parity results of double and triple zeta values. As another important application, we give certain refinement of restricted sum formulas, which gives restricted sum formulas among MZVs of an arbitrary depth $r$ which were previously known only in the cases of depth $2,3,4$. 
Furthermore, considering a sub-root system of type $B_r$ analogously, we can give relevant analogues of the Hoffman-Zagier formula, parity results and restricted sum formulas.
\end{abstract}

\maketitle

\section{Introduction}\label{sec-1}

Let $\mathbb{N}$, $\mathbb{N}_0$, $\mathbb{Z}$, $\mathbb{Q}$, $\mathbb{R}$, 
$\mathbb{C}$ be the set of positive integers, non-negative integers, rational integers,
rational numbers, real numbers, and complex numbers, respectively.

We define the Euler-Zagier $r$-ple zeta-function (simply called the Euler-Zagier sum) by 
\begin{equation}
\zeta_r(s_1,\ldots,s_r)=\sum_{0<n_1<\cdots<n_r}\frac{1}{n_1^{s_1}n_2^{s_2}
    \cdots n_r^{s_r}},\label{e-1-1}
\end{equation}
where $s_1,\ldots,s_r$ are complex variables. When $(s_1,s_2,\ldots,s_r)\in \mathbb{N}^r$ $(s_r>1)$, this is called the multiple zeta value (MZV) of depth $r$ first studied by Hoffman \cite{Hoff} and Zagier\cite{Za}. 
Though the opposite order of summation in the 
definition of $\zeta_{r}(s_1,\ldots,s_r)$ is also used recently, we use the order in \eqref{e-1-1} through this paper because it is natural in our study. 
In the research of MZVs, the main target is to give non-trivial relations  
among them, in order to investigate the structure of the algebra generated by them 
(for the details, see Kaneko \cite{Ka}).

In our previous papers \cite{KMT}-\cite{KM4} and \cite{MT1}, as more general multiple series, 
we defined and studied multi-variable zeta-functions associated with root systems of 
type $X_r$ ($X=A,B,C,D,E,F,G$) denoted by 
$\zeta_r(s_1,\ldots,s_n;X_r)$ where $n$ is the number of positive roots of type $X_r$ 
(see definition \eqref{def-zeta}). In particular when $s_1=\cdots=s_r=s$, 
$\zeta_r(s,\ldots,s;X_r)$ essentially coincides with the Witten zeta-function 
(see Witten \cite{Wi} and Zagier\cite{Za}). An important fact is 
\begin{equation}
\zeta_r(2k,2k,\ldots,2k;X_r)\in \mathbb{Q}\cdot \pi^{2kn}\quad (k\in \mathbb{N}), 
\label{Witten-VF}
\end{equation}
which is a consequence of Witten's volume formula given in \cite{Wi}. Since we 
considered multi-variable version of Witten zeta-function, we were able to determine
the rational coefficients in \eqref{Witten-VF} explicitly in a generalized form 
(see \cite[Thoerem 4.6]{KM3}). 

Recently, in our previous paper \cite{KMT-MZ}, we regarded MZVs as special values of 
zeta-functions of root systems of type $A_r$, and clarified the structure of the 
shuffle product 
procedure for MZVs from this viewpoint.    In fact, we showed that the shuffle product
procedure can be described in terms of partial fraction decompositions of
zeta-functions of root systems of type $A_r$.

The main idea in the present paper is to regard \eqref{e-1-1} as a specialization of 
zeta-functions of root
systems of type $C_r$ (see below). 
It is essential in our theory that $C_r$ is not simply-laced.
In fact, there exists a subset of the root system of 
type $C_r$ so that the Euler-Zagier sum \eqref{e-1-1} is the zeta-function 
associated with this subset (see Section \ref{sec-4}). 
This subset itself is a root system, and hence the Weyl group 
naturally acts on \eqref{e-1-1}.  General fundamental results will be stated in Section 
\ref{sec-3}, and their proofs will be given in Section \ref{sec-proof1}.
As a consequence, it can be shown that a kind of formula \eqref{Witten-VF} 
corresponding to this sub-root system implies the well-known result given by Hoffman \cite[Section 2]{Hoff} and Zagier \cite[Section 9]{Za} independently: 
\begin{equation}
\zeta_r(2k,2k,\ldots,2k)\in \mathbb{Q}\cdot \pi^{2kr}\quad (k\in \mathbb{N})\label{H-Z}
\end{equation}
(see Corollary \ref{Cor-Z}). 

Furthermore, based on this observation in the cases when $r=2,3$, we will give explicit 
formulas for double series and for triple series (see Proposition \ref{Pr-1} and 
Theorem \ref{T-5-1}) which include what is called the parity results for double and 
triple zeta values (see Corollary \ref{C-5-3}). 

Similarly we can consider analogues of those results corresponding to a sub-root system of type $B_r$. In fact, we can define a $B_r$-type analogue of $\zeta_{r}({\bf s})$ by 
\begin{align}
& \zeta_{r}^\sharp({\bf s}) =\sum_{m_1,\ldots,m_r=1}^\infty
\prod_{i=1}^{r}
\frac{1}{(2\sum_{j=r-i+1}^{r-1}m_j+m_r)^{{s_i}}}, \label{def-Br-Zeta}
\end{align}
which is a ``partial sum'' of the series of $\zeta_{r}({\bf s})$ (see Section \ref{sec-6}). From the viewpoint of root systems, we see that this has some properties similar to those of $\zeta_{r}({\bf s})$, because the root system of type $B_r$ is a dual of that of type $C_r$. Actually we can obtain an analogue of \eqref{H-Z} for this series (see Corollary \ref{C-6-2}). We also prove a formula between the values of $\zeta_{2}^\sharp({\bf s})$ and the Riemann zeta values (see Theorem \ref{T-B2-EZ}), which gives the parity result corresponding to type $B_r$ (see Theorem \ref{T-B2-EZ}). This result plays an important role in a recent study on the dimension of the linear space spanned by double zeta values of level $2$ given by Kaneko and Tasaka (see \cite{Ka-Ta}).

The fact that parity results hold in those classes implies that those are ``nice''
classes.   In Section \ref{sec-acs} we will study those classes from the analytic
point of view, and prove that those classes, as well as the subclass of
zeta-functions of root systems of type $A_r$ introduced in \cite{KMT-MZ},
are ``closed'' in a certain analytic sense.

Another important consequence of our fundamental theorem in Section \ref{sec-3} is
the ``refined restricted sum formulas'' for the values of $\zeta_{r}({\bf s})$
and $\zeta_{r}^\sharp({\bf s})$, which are embodied in Corollaries \ref{Cor-Cr-Sr}  
and \ref{Cor-Br-Sr}.
One of the famous formulas among MZVs is the sum formula, which is, in the case of 
double zeta values, written as
\begin{equation}
\sum_{j=2}^{K-1}\zeta_2(K-j,j)=\zeta(K)\quad (K\in \mathbb{Z}_{\geq 3}).  \label{sumformula}
\end{equation}
Gangl, Kaneko and Zagier \cite{GKZ} obtained the following formulas, which
``divide'' \eqref{sumformula} for even $K$ into two parts: 
\begin{equation}
\begin{split}
& \sum_{a,b \in \mathbb{N}\atop a+b=N} \zeta_2(2a,2b)=\frac{3}{4}\zeta(2N)\in \mathbb{Q}\cdot \pi^{2N}\quad (N\in \mathbb{Z}_{\geq 2}), \\
& \sum_{a,b \in \mathbb{N}\atop a+b=N} \zeta_2(2a-1,2b+1)=\frac{1}{4}\zeta(2N)\in \mathbb{Q}\cdot \pi^{2N}\quad (N\in \mathbb{Z}_{\geq 2}),
\end{split}
\label{F-GKZ}
\end{equation}
which are sometimes called the restricted sum formulas. 
More recently, Shen and Cai \cite{Shen-Cai} gave restricted sum formulas for triple and fourth zeta values (see \eqref{sumf-triple} and \eqref{sumf-fourth}).
As we will discuss in Section \ref{sumf},
our Corollaries \ref{Cor-Cr-Sr} and \ref{Cor-Br-Sr} give more refined restricted
sum formulas for $\zeta_r({\bf s})$ and for $\zeta_r^\sharp({\bf s})$ of an 
arbitrary depth $r$.   From these refined formulas we can deduce the restricted
sum formulas for an arbitrary depth $r$,
actually in a generalized form involving a parameter $d$
(see Theorems \ref{sumf-EZ-Cr} and \ref{sumf-EZ-Br}).

A part of the results in the present paper has been announced in \cite{KMT-PJA}.


\section{Zeta-functions of root systems and root sets}\label{sec-2}

In this section, we recall the definition of zeta-functions of root systems studied in our papers \cite{KMT}-\cite{KM3}. 
For the details of basic facts about root systems and Weyl groups, see                   
\cite{Bourbaki,Hum72,Hum}.

Let $V$ be an $r$-dimensional real vector space equipped with an inner product $\langle \cdot,\cdot\rangle$.
The dual space $V^*$ is identified with $V$ via the inner product of $V$.
Let $\Delta$ be a finite irreducible reduced root system, and
$\fs=\{\alpha_1,\ldots,\alpha_r\}$ its fundamental system.
We fix 
$\Delta_+$ and $\Delta_-$ as the set of all positive roots and negative roots respectively.
Then we have a decomposition of the root system $\Delta=\Delta_+\coprod\Delta_-$.
Let $Q=Q(\Delta)$ be the root lattice, $Q^\vee$ the coroot lattice,
$P=P(\Delta)$ the weight lattice, $P^\vee$ the coweight lattice,
and
$P_{++}$ the set of integral strongly dominant weights
respectively defined by
\begin{align*}
&  Q=\bigoplus_{i=1}^r\mathbb{Z}\,\alpha_i,\qquad
  Q^\vee=\bigoplus_{i=1}^r\mathbb{Z}\,\alpha^\vee_i,\\
& P=\bigoplus_{i=1}^r\mathbb{Z}\,\lambda_i, \qquad 
  P^\vee=\bigoplus_{i=1}^r\mathbb{Z}\,\lambda^\vee_i,\qquad P_{++}=\bigoplus_{i=1}^r\mathbb{N}\,\lambda_i,
\end{align*}
where the fundamental weights $\{\lambda_j\}_{j=1}^r$
and
the fundamental coweights $\{\lambda_j^\vee\}_{j=1}^r$
are the dual bases of $\fs^\vee$ and $\fs$
satisfying $\langle \alpha_i^\vee,\lambda_j\rangle=\delta_{ij}$ (Kronecker's delta)
and $\langle \lambda_i^\vee,\alpha_j\rangle=\delta_{ij}$ respectively.

Let $\sigma_\alpha :V\to V$ be the reflection with respect to a root $\alpha\in\Delta$ defined by 
$$\sigma_\alpha:v\mapsto v-\langle \alpha^\vee,v\rangle\alpha.$$
For a subset $A\subset\Delta$, let
$W(A)$ be the group generated by reflections $\sigma_\alpha$ for all $\alpha\in A$. In particular, $W=W(\Delta)$ is the Weyl group, and 
$\{\sigma_j:=\sigma_{\alpha_j}\,|\,1\leq j \leq r\}$ generates $W$. 
For $w\in W$, denote 
$\Delta_w=\Delta_+\cap w^{-1}\Delta_-$.
The zeta-function associated with $\Delta$ is defined by
\begin{equation}
  \zeta_r(\mathbf{s},\mathbf{y};\Delta)
  =
  \sum_{\lambda\in P_{++}}
  e^{2\pi i\langle\mathbf{y},\lambda\rangle}\prod_{\alpha\in\Delta_+}
  \frac{1}{\langle\alpha^\vee,\lambda\rangle^{s_\alpha}}, \label{def-zeta}
\end{equation}
where $\mathbf{s}=(s_{\alpha})_{\alpha\in\Delta_+}\in \mathbb{C}^{|\Delta_+|}$ 
and $\mathbf{y}\in V$.
This can be regarded as a 
multi-variable version of Witten zeta-functions formulated by Zagier \cite{Za} 
based on the work of Witten \cite{Wi}. 

Let $\Delta^*$ be a subset of $\Delta_+$.   We call $\Delta^*$ a
{\it root set} (or a {\it root subset} of $\Delta_+$)
if, for any $\lambda_j$ ($1\leq j\leq r$), there exists an element
$\alpha\in\Delta^*$ for which $\langle\alpha,\lambda_j \rangle\neq 0$ holds.
We define the zeta-function associated with a root set $\Delta^*$ by
\begin{equation}                                                                       
  \zeta_r(\mathbf{s},\mathbf{y};\Delta^*)                                               
  =                                                                                     
  \sum_{\lambda\in P_{++}}                                                              
  e^{2\pi i\langle\mathbf{y},\lambda\rangle}\prod_{\alpha\in\Delta^*}                   
  \frac{1}{\langle\alpha^\vee,\lambda\rangle^{s_\alpha}}. \label{def-zeta-root-set}
\end{equation}
In the case $\mathbf{y}=\mathbf{0}$, this zeta-function was introduced in \cite{KM2}.
When the root system is of type $X_r$, we write $\Delta=\Delta(X_r)$, 
$\Delta^*=\Delta^*(X_r)$, and so on.

\begin{remark}
The notion of $\zeta_r(\mathbf{s},\mathbf{y};\Delta^*)$ depends not only on
$\Delta^*$, but also on $\Delta_+$, because the summation on
\eqref{def-zeta-root-set} runs over all strongly dominant weights associated
with $\Delta_+$.
\end{remark}

\section{Fundamental formulas}\label{sec-3}

In this section, we state several fundamental formulas which are certain extensions 
of our previous results given in \cite{KM2,KM5,KM3}. 
Proofs of theorems stated in this section will be given in Section \ref{sec-proof1}.

Let
$\mathscr{V}$ be the set of all bases $\mathbf{V}\subset\Delta_+$.
Let
$\mathbf{V}^*=\{\mu_\beta^{\mathbf{V}}\}_{\beta\in\mathbf{V}}$ be the dual basis of $\mathbf{V}^\vee=\{\beta^\vee\}_{\beta\in\mathbf{V}}$.
Let
$L(\mathbf{V}^\vee)
=\bigoplus_{\beta\in\mathbf{V}}\mathbb{Z}\beta^\vee$. Then
we have $\abs{Q^\vee/L(\mathbf{V}^\vee)}<\infty$.
Fix $\phi\in V$ such that $\langle\phi,\mu^{\mathbf{V}}_\beta\rangle\neq 0$ for all $\mathbf{V}\in\mathscr{V}$ and $\beta\in\mathbf{V}$.
If the root system $\Delta$ is of $A_1$ type, then we choose $\phi=\alpha_1^\vee$.
We define a multiple generalization of the fractional part as
\begin{equation*}
\{\mathbf{y}\}_{\mathbf{V},\beta}=
\begin{cases}
  \{\langle\mathbf{y},\mu^{\mathbf{V}}_\beta\rangle\}\quad&(\langle\phi,\mu^{\mathbf{V}}_\beta\rangle>0),\\
  1-\{-\langle\mathbf{y},\mu^{\mathbf{V}}_\beta\rangle\}&(\langle\phi,\mu^{\mathbf{V}}_\beta\rangle<0),
\end{cases}
\end{equation*}
where the notation $\{x\}$ on the right-hand sides stands for the usual fractional 
part of $x\in\mathbb{R}$.
Let $\mathbf{T}=\{t\in\mathbb{C}\;|\;|t|<2\pi\}^{|\Delta_+|}$.

\begin{definition}
  \label{thm:exp_F}
  For $\mathbf{t}=(t_{\alpha})_{\alpha\in\Delta_+}\in\mathbf{T}$ and
$\mathbf{y}\in V$,
  we define
  \begin{equation}\label{def-F}
    \begin{split}
      F&(\mathbf{t},\mathbf{y};\Delta)=
      \sum_{\mathbf{V}\in\mathscr{V}}
      \Bigl(
      \prod_{\gamma\in\Delta_+\setminus\mathbf{V}}
      \frac{t_\gamma}
      {t_\gamma-\sum_{\beta\in\mathbf{V}}t_\beta\langle\gamma^\vee,\mu^{\mathbf{V}}_\beta\rangle}
      \Bigr)
      \\
      &\times
      \frac{1}{\abs{Q^\vee/L(\mathbf{V}^\vee)}}
      \sum_{q\in Q^\vee/L(\mathbf{V}^\vee)}
      \Bigl(
      \prod_{\beta\in\mathbf{V}}\frac{t_\beta\exp
        (t_\beta\{\mathbf{y}+q\}_{\mathbf{V},\beta})}{e^{t_\beta}-1}
      \Bigr)
      \\
      &=
      \sum_{\mathbf{k}\in\mathbb{N}_{0}^{|\Delta_+|}}
      \mathcal{P}(\mathbf{k},\mathbf{y};\Delta)
      \prod_{\alpha\in\Delta_+} \frac{t_\alpha^{k_\alpha}}{k_\alpha!}
    \end{split}
  \end{equation}
which is independent of choice of $\phi$.
\end{definition}

\begin{remark}
In \cite{KM5}, $F(\mathbf{t},\mathbf{y};\Delta)$ is defined in a different way.
The above is \cite[Theorem 4.1]{KM5}. In particular when $\Delta=\Delta(A_1)$, we see that 
$$F(\mathbf{t},\mathbf{y};\Delta(A_1))=\frac{te^{t\{y\}}}{e^t-1},$$
which is the generating function of ordinary Bernoulli periodic functions $\{B_k(\{y\})\}$.
\end{remark}

Let
\begin{equation}
  \label{eq:def_S}
  S(\mathbf{s},\mathbf{y};\Delta)
  =\sum_{\lambda\in P\setminus H_{\Delta^\vee}}
  e^{2\pi i\langle \mathbf{y},\lambda\rangle}
  \prod_{\alpha\in\Delta_+}
  \frac{1}{\langle\alpha^\vee,\lambda\rangle^{s_\alpha}},
\end{equation}
where $H_{\Delta^\vee}=\{v\in V~|~\langle \alpha^\vee,v\rangle=0\text{
  for some }\alpha\in\Delta\}$ is the set of all walls of Weyl
chambers.  For $\mathbf{s}\in\mathbb{C}^{|\Delta_+|}$, we define 
$(w\mathbf{s})_{\alpha}=s_{w^{-1}\alpha}$,
where if $w^{-1}\alpha\in\Delta_-$ we use the convention $s_{-\alpha}=s_\alpha$.
\begin{prop}[{\cite[Theorem 4.4]{KM3},\cite[Proposition 3.2]{KM5}}]
  \label{prop:ZP}
  \begin{equation}
    \label{eq:formula1}
    \begin{split}
      S(\mathbf{k},\mathbf{y};\Delta)
      &=
      \sum_{w\in W}
      \Bigl(\prod_{\alpha\in\Delta_+\cap w\Delta_-}
      (-1)^{k_{\alpha}}\Bigr)
      \zeta_r(w^{-1}\mathbf{k},w^{-1}\mathbf{y};\Delta)\\
      &=(-1)^{\abs{\Delta_+}}  
      \mathcal{P}(\mathbf{k},\mathbf{y};\Delta)
      \biggl(\prod_{\alpha\in\Delta_+}
      \frac{(2\pi i)^{k_\alpha}}{k_\alpha!}\biggr)
    \end{split}
  \end{equation}
  for $k_\alpha\in\mathbb{Z}_{\geq2}$ ($\alpha\in\Delta_+$). 
\end{prop}
\begin{remark}
  It should be noted that 
  the formula \eqref{eq:formula1} holds in the cases
  $k_\alpha=1$ for some $\alpha\in\Delta_+$, while it does not hold in the cases
  $k_\alpha=0$ for any $\alpha\in\Delta_+$.
\end{remark}

For $\mathbf{v}\in V$, and a differentiable function $f$ on $V$,
let 
\begin{equation*}
 (\partial_{\mathbf{v}}f)(\mathbf{y})=\lim_{h\to 0}\frac{f(\mathbf{y}+h\mathbf{v})-f(\mathbf{y})}{h}
\end{equation*}
and for $\alpha\in\Delta_+$,
\begin{equation*}
  \mathfrak{D}_\alpha=
  \frac{\partial}{\partial t_\alpha}
  \biggr\rvert_{t_\alpha=0}\partial_{\alpha^\vee}.
\end{equation*}
Let $\Delta^*\subset\Delta_+$ be a root set and
let $A=\Delta_+\setminus\Delta^*=\{\nu_1,\ldots,\nu_N\}\subset\Delta_+$, and define
\begin{equation*}                                                                    
\mathfrak{D}_A=\mathfrak{D}_{\nu_N}\cdots \mathfrak{D}_{\nu_1}.                       
\end{equation*} 
Similarly we define
\begin{gather}
  \mathfrak{D}_{\alpha,2}=
  \frac{1}{2}\frac{\partial^2}{\partial t_\alpha^2}
  \biggr\rvert_{t_\alpha=0}\partial^2_{\alpha^\vee},\\
  \mathfrak{D}_{A,2}=\mathfrak{D}_{\nu_N,2}\cdots \mathfrak{D}_{\nu_1,2}.                       
\end{gather}

Further, let $A_j=\{\nu_1,\ldots,\nu_j\}$ ($1\leq j\leq N-1$), $A_0=\emptyset$, and
\begin{equation*}
  \mathscr{V}_A=\{\mathbf{V}\in\mathscr{V}~|~
  \text{$\nu_{j+1}\notin{\rm L.h.}[\mathbf{V}\cap A_j]\;(0\leq j \leq N-1)$} \},
\end{equation*}
where ${\rm L.h.}[\;\cdot\;]$ denotes the linear hull (linear span).
Let $\mathscr{R}$ be the set of all linearly independent subsets
$\mathbf{R}=\{\beta_1,\ldots,\beta_{r-1}\}\subset\Delta$
and
\begin{equation}
  \label{eq:def_H_R}
  \mathfrak{H}_{\mathscr{R}}:=
  \bigcup_{\substack{\mathbf{R}\in\mathscr{R}\\q\in Q^\vee}}({\rm L.h.}
[\mathbf{R}^\vee]+q).
\end{equation}

\begin{remark}\label{tsuika}
It is to be noted that $\mathbf{y}\in\mathfrak{H}_{\mathscr{R}}$ if and only if
$\langle\mathbf{y}+q,\mu^{\mathbf{V}}_\beta\rangle\in\mathbb{Z}$ for some 
$\mathbf{V}\in\mathscr{V},\mathbf{\beta}\in\mathbf{V},q\in Q^\vee$.
In fact, if $\mathbf{y}\in\mathfrak{H}_{\mathscr{R}}$ then we can write
$\mathbf{y}=\sum_{j=1}^{r-1}a_j \beta_j^{\vee}+q$ ($a_j\in\mathbb{R}$).
We can find an element $\beta_r\in\Delta$ such that
$\mathbf{V}=\{\beta_1,\ldots,\beta_r\}\in\mathscr{V}$.
Then $\langle\mathbf{y}-q,\mu_{\beta_r}^{\mathbf{V}}\rangle=0\in\mathbb{Z}$.
Conversely, assume $\langle\mathbf{y}+q,\mu^{\mathbf{V}}_\beta\rangle=c\in\mathbb{Z}$.
Write $\mathbf{V}=\{\beta_1,\ldots,\beta_{r-1},\beta\}$.   Since this is a basis, we
may write $\mathbf{y}+q=\sum_{j=1}^{r-1}a_j \beta_j^{\vee}+a\beta^{\vee}$ with
$a_j,a\in\mathbb{R}$.   Then 
$c=\langle\mathbf{y}+q,\mu^{\mathbf{V}}_\beta\rangle=a$, especially $a\in\mathbb{Z}$.
Therefore $a\beta^{\vee}-q\in Q^{\vee}$, which implies 
$\mathbf{y}\in \mathfrak{H}_{\mathscr{R}}$.
\end{remark}

\begin{definition}
  For $\Delta_+\setminus \Delta^*=A=\{\nu_1,\ldots,\nu_N\}\subset\Delta_+$, $\mathbf{t}_{\Delta^*}=\{t_\alpha\}_{\alpha\in\Delta^*}$ and $\mathbf{y}\in V$,
  we define
  \begin{equation}
    \begin{split}
      F_{\Delta^*}(\mathbf{t}_{\Delta^*},\mathbf{y};\Delta)
      &=
      \sum_{\mathbf{V}\in\mathscr{V}_A}
      (-1)^{\abs{A\setminus\mathbf{V}}}\\
      & \times\Bigl(
      \prod_{\gamma\in\Delta_+\setminus(\mathbf{V}\cup A)}
      \frac{t_\gamma}
      {t_\gamma-\sum_{\beta\in\mathbf{V}\setminus A}t_\beta\langle\gamma^\vee,\mu^{\mathbf{V}}_\beta\rangle}
      \Bigr)
      \\
      &\times
      \frac{1}{\abs{Q^\vee/L(\mathbf{V}^\vee)}}
      \sum_{q\in Q^\vee/L(\mathbf{V}^\vee)}
      \Bigl(
      \prod_{\beta\in\mathbf{V}\setminus A}\frac{t_\beta\exp
        (t_\beta\{\mathbf{y}+q\}_{\mathbf{V},\beta})}{e^{t_\beta}-1}
      \Bigr).
    \end{split}
  \end{equation}
\end{definition}

\begin{thrm}
  \label{thm:main0}
  For $\Delta_+\setminus \Delta^*=A=\{\nu_1,\ldots,\nu_N\}\subset\Delta_+$, $\mathbf{t}_{\Delta^*}=\{t_\alpha\}_{\alpha\in\Delta^*}$ and $\mathbf{y}\in V\setminus\mathfrak{H}_{\mathscr{R}}$.
  we have
  \begin{equation}
    \label{eq:main0}
    \bigl(\mathfrak{D}_A F\bigr) (\mathbf{t}_{\Delta^*},\mathbf{y};\Delta)=
    \bigl(\mathfrak{D}_{A,2} F\bigr) (\mathbf{t}_{\Delta^*},\mathbf{y};\Delta)=
    F_{\Delta^*}(\mathbf{t}_{\Delta^*},\mathbf{y};\Delta),
  \end{equation}
  and hence is independent of choice of the order of $A$.
  The function $F_{\Delta^*}(\mathbf{t}_{\Delta^*},\mathbf{y};\Delta)$
  is the continuous extension 
  of $\bigl(\mathfrak{D}_A F\bigr) (\mathbf{t}_{\Delta^*},\mathbf{y};\Delta)$
  in $\mathbf{y}$ in the sense that 
  $\bigl(\mathfrak{D}_A F\bigr) (\mathbf{t}_{\Delta^*},\mathbf{y}+c\phi;\Delta)$
  tends continuously to $F_{\Delta^*}(\mathbf{t}_{\Delta^*},\mathbf{y};\Delta)$  
  when $c\to 0+$, 
  and is holomorphic with respect to $\mathbf{t}_{\Delta^*}$ around the origin.
\end{thrm}

\begin{definition}
For $\Delta^*\subset\Delta_+$ and $\mathbf{t}_{\Delta^*}=\{t_\alpha\}_{\alpha\in\Delta^*}$, we define $\mathcal{P}_{\Delta^*}(\mathbf{k}_{\Delta^*},\mathbf{y};\Delta)$ by 
  \begin{align*}
    &
    F_{\Delta^*}(\mathbf{t}_{\Delta^*},\mathbf{y};\Delta)
    \\
    &=
    \sum_{\mathbf{k}_{\Delta^*}\in \mathbb{N}_{0}^{\abs{\Delta^*}}}\mathcal{P}_{\Delta^*}(\mathbf{k}_{\Delta^*},\mathbf{y};\Delta)
    \prod_{\alpha\in\Delta^*}
    \frac{t_{\alpha}^{k_\alpha}}{k_\alpha!}.
  \end{align*}
\end{definition}

\begin{thrm}
\label{thm:main1}
  For
  $\mathbf{s}=\mathbf{k}=(k_\alpha)_{\alpha\in\Delta_+}$ with
  $k_\alpha\in\mathbb{Z}_{\geq2}$ ($\alpha\in \Delta^*$),
  $k_\alpha=0$ ($\alpha\in \Delta_+\setminus \Delta^*$),
  we have
  \begin{align}
    &\sum_{w\in W}
    \Bigl(\prod_{\alpha\in\Delta_+\cap w\Delta_-}
    (-1)^{k_{\alpha}}\Bigr)
    \zeta_r(w^{-1}\mathbf{k},w^{-1}\mathbf{y};\Delta) \label{Th-main}\\
    & =(-1)^{\abs{\Delta_+}}  
    \mathcal{P}_{\Delta^*}(\mathbf{k}_{\Delta^*},\mathbf{y};\Delta)
    \biggl(\prod_{\alpha\in\Delta^*}
    \frac{(2\pi i)^{k_\alpha}}{k_\alpha!}\biggr)
   \notag
  \end{align}    
provided all the series on the left-hand side absolutely converge. 
\end{thrm}

Assume that $\Delta$ is not simply-laced.  Then we have the disjoint
union $\Delta=\Delta_l\cup\Delta_s$, where $\Delta_l$ is the set of
all long roots and $\Delta_s$ is the set of all short roots.

Note that if there is an odd $k_i$, then both hand sides vanish in \eqref{Th-main}. On the other hand, when all $k_i's$ are even, 
by applying Theorem \ref{thm:main1} to $\Delta^*=\Delta_l$ or $\Delta_s$, we obtain
the following theorem immediately, which is a generalization of the explicit
volume formula proved in \cite[Theorem 4.6]{KM3}.

\begin{thrm}
\label{thm:main2}
Let $\Delta_1=\Delta_l$ (resp.~$\Delta_s$), $\Delta_2=\Delta_s$ (resp.~$\Delta_l$),
and $\Delta_{j+}=\Delta_j\cap\Delta_+$ ($j=1,2$).
Then $\Delta_{j+}$ ($j=1,2$) is a root subset of $\Delta_+$.
For
  $\mathbf{s}_1=\mathbf{k}_1=(k_\alpha)_{\alpha\in\Delta_{1+}}$ with
  $k_\alpha=k\in 2\mathbb{N}$ (for all $\alpha\in\Delta_{1+}$) 
and $\nu\in P^\vee/Q^\vee$, we have
  \begin{align}
    &\zeta_r(\mathbf{k}_1,\nu;\Delta_{1+}) =\frac{(-1)^{\abs{\Delta_+}}}{|W|}  
    \mathcal{P}_{\Delta_{1+}}(\mathbf{k}_1,\nu;\Delta)
    \biggl(\prod_{\alpha\in\Delta_{1+}}
    \frac{(2\pi i)^{k_\alpha}}{k_\alpha!}\biggr).\label{main2}
  \end{align}    
\end{thrm}

\begin{remark}
Let $\mathbf{s}=\mathbf{k}=(k_\alpha)_{\alpha\in\Delta_{+}}$ with
$k_{\alpha}=k\in 2\mathbb{N}$ ($\alpha\in\Delta_{1+}$) and
$k_\alpha=0$ ($\alpha\in\Delta_{2+}$).   Then obviously
$\zeta_r(\mathbf{k}_1,\nu;\Delta_{1+})=\zeta_r(\mathbf{k},\nu;\Delta)$.
Our proof of Theorem \ref{thm:main2} is actually based on the latter viewpoint.
\end{remark}

\section{Multiple zeta values and zeta-functions of root system of type $C_r$}
\label{sec-4}

Now we study MZVs from the viewpoint of zeta-functions of root systems of type $C_r$. 
For $\Delta=\Delta(C_r)$, we have the disjoint union $\Delta_+^\vee=(\Delta_{l+})^\vee\cup(\Delta_{s+})^\vee$,
where $\Delta_{l+}=\Delta_{l+}(C_r)=\Delta_l(C_r)\cap\Delta_+(C_r)$, 
$\Delta_{s+}=\Delta_{s+}(C_r)=\Delta_s(C_r)\cap\Delta_+(C_r)$, and 
\begin{align*}
  (\Delta_{l+})^\vee&
  =\{
  \alpha_r^\vee,\
  \alpha_{r-1}^\vee+\alpha_r^\vee,\
  \alpha_{r-2}^\vee+\alpha_{r-1}^\vee+\alpha_r^\vee,\
  \ldots,  
  \alpha_1^\vee+\cdots+\alpha_r^\vee
\}.
\end{align*}
Since $P^\vee/Q^\vee=\{\mathbf{0},\lambda_r^\vee\}$,
Therefore, for $\mathbf{s}_l=(s_{\alpha})_{\alpha\in\Delta_{l+}}$, we have
\begin{align*}
  \zeta_r(\mathbf{s}_l,\mathbf{0};\Delta_{l+}(C_r))&=\sum_{m_1,\ldots,m_r=1}^\infty
\prod_{i=1}^{r}
\frac{1}{(\sum_{j=r-i+1}^{r-1}m_j+m_r)^{{s_i}}},\\
  \zeta_r(\mathbf{s}_l,\lambda_r^\vee;\Delta_{l+}(C_r))&=\sum_{m_1,\ldots,m_r=1}^\infty
\prod_{i=1}^{r}
\frac{(-1)^{m_r}}{(\sum_{j=r-i+1}^{r-1}m_j+m_r)^{{s_i}}},
\end{align*}
where the first equation is exactly the Euler-Zagier sum $\zeta_{r}(s_1,\ldots,s_r)$ (see \eqref{e-1-1}). 
%
In order to apply Theorems \ref{thm:main1} and \ref{thm:main2} to MZVs,
we rewrite the root system of type $C_r$ in terms of standard orthonormal basis $\{e_1,\ldots,e_r\}$.
We put $\alpha_i^\vee=e_{i}-e_{i+1}$ for $1\leq i\leq r-1$ and $\alpha_r^\vee=e_r$.
Then we have
\begin{equation*}
  (\Delta_{l+})^\vee
  =\{
\alpha_r^\vee=e_r,\
\alpha_{r-1}^\vee+\alpha_r^\vee=e_{r-1},\
\alpha_{r-2}^\vee+\alpha_{r-1}^\vee+\alpha_r^\vee=e_{r-2},\
\ldots,  
\alpha_1^\vee+\cdots+\alpha_r^\vee=e_1
\}.
\end{equation*}
In this realization, we see that
$W(C_r)=(\mathbb{Z}/2\mathbb{Z})^r\rtimes \mathfrak{S}_r$, where
$\mathfrak{S}_r$ is the symmetric group of degree $r$ which permutes
bases, and the $j$-th $\mathbb{Z}/2\mathbb{Z}$ flips the sign of
$e_j$. Since the sign flips act trivially on the variables
$\mathbf{s}_l$, from Theorem \ref{thm:main1} we obtain 
the following formulas.   These are the ``refined restricted sum formulas'' for
$\zeta_r(\mathbf{s})$, which we will discuss in Section \ref{sumf}.
\begin{cor} \label{Cor-Cr-Sr}
Let $\Delta=\Delta(C_r)$. For $(2{\bf k})_l=(2k_{\alpha})_{\alpha\in\Delta_{l+}}=(2k_1,\ldots,2k_r)\in \left(2\mathbb{N}\right)^r$ and 
$\mathbf{y}=\nu\in P^\vee/Q^\vee$, 
\begin{equation} \label{EZ-Sr-1}
  \sum_{\sigma\in\mathfrak{S}_r}
  \zeta_r(\sigma^{-1}(2\mathbf{k})_l,\nu;\Delta_{l+})
=\frac{(-1)^{r}}{2^r}\mathcal{P}_{\Delta_{l+}}
((2\mathbf{k})_l,\nu;\Delta)
  \prod_{j=1}^{r}\frac{(2\pi i)^{2k_j}}{(2k_j)!}  \in\mathbb{Q}\cdot \pi^{2\sum_{j=1}^{r}k_j}. 
\end{equation}
In particular when $\nu={\bf 0}$, 
\begin{equation} \label{EZ-Sr-11}
  \sum_{\sigma\in\mathfrak{S}_r}
  \zeta_r(2k_{\sigma^{-1}(1)},\ldots,2k_{\sigma^{-1}(r)})
=\frac{(-1)^{r}}{2^r}\mathcal{P}_{\Delta_{l+}}
((2\mathbf{k})_l,{\bf 0};\Delta)
  \prod_{j=1}^{r}\frac{(2\pi i)^{2k_j}}{(2k_j)!}  \in\mathbb{Q}\cdot \pi^{2\sum_{j=1}^{r}k_j}. 
\end{equation}

\end{cor}
%
%
%
Also Theorem \ref{thm:main2} in the case of type $C_r$ immediately gives the following.

\begin{cor}\label{Cor-Z}
Let $\Delta=\Delta(C_r)$. For $({\bf 2k})_l=(2k,\ldots,2k)$ with any $k\in \mathbb{N}$, 
\begin{align}
  & \zeta_{r}(2k,2k,\ldots,2k)=\frac{(-1)^{r}}{2^r r!}\mathcal{P}_{\Delta_{l+}}
((\mathbf{2k})_l,{\bf 0};\Delta)
  \frac{(2\pi i)^{2kr}}{\{ (2k)! \}^r} 
\in\mathbb{Q}\cdot \pi^{2kr}. \label{Zagier-F2}
\end{align}
\end{cor}

\begin{remark}\label{Rem-Hof}
The fact 
$$\sum_{\sigma\in\mathfrak{S}_r} \zeta_r(2k_{\sigma^{-1}(1)},\ldots,2k_{\sigma^{-1}(r)})\in\mathbb{Q}\cdot \pi^{2\sum_{j=1}^{r}k_j}$$
was first proved by Hoffman \cite[Theorem 2.2]{Hoff}. This gives the well-known result 
$$\zeta_{r}(2k,\ldots,2k)\in\mathbb{Q}\cdot \pi^{2kr},$$
which was also given by Zagier \cite[Section 9]{Za} independently. 
Broadhurst, Borwein and Bradley gave explicit formulas for these values in \cite[Section 2]{BBB}. 
Also it is known that 
\begin{equation}
\zeta_{r}(2k,\ldots,2k)=\cc_r^{(k)}\frac{(2\pi i)^{2kr}}{(2kr)!},\label{Zagier-F}
\end{equation}
where
$$\cc_0^{(k)}=1,\ \ \cc_{n}^{(k)}=\frac{1}{2n}\sum_{j=1}^{n}(-1)^j \binom{2nk}{2jk}B_{2jk}\cc_{n-j}^{(k)}\ \ (n \geq 1).$$
Formula \eqref{Zagier-F} was first published in the lecture notes \cite{AK1}, \cite{AK2} 
written in Japanese (Exercise 5, Section 1.1 of those lecture notes). 
See also Muneta \cite{Mun}.  

We emphasize that \eqref{Zagier-F} can be regarded as a kind of Witten's volume formula \eqref{Zagier-F2}. 
Because \eqref{Zagier-F2} and \eqref{Zagier-F} in the case $r=1$ are both Euler's
well-known formula
\begin{equation}
\zeta(2k)=-B_{2k}\frac{(2\pi i)^{2k}}{2(2k)!}\qquad (k\in \mathbb{N}), \label{Euler-F}
\end{equation}
we can see that 
$\mathcal{P}_{\Delta_{l+}}((\mathbf{2k})_l,{\bf 0};\Delta)$ and $\cc_r^{(k)}$ are different types of generalizations of the ordinary Bernoulli number $B_{2k}$.
\end{remark}

\begin{example} \label{Exam-C2}
Let $\Delta=\Delta(C_2)$ be the root system of type $C_2$. 
By Theorem \ref{thm:main0}, we have
  \begin{align*}
      \ (\mathfrak{D}_{\Delta_{s+}}F)(t_1,t_2,y_1,y_2;\Delta)
      & =1+\frac{t_1 t_2 e^{\{y_2\}t_1}}{(e^{t_1}-1)(t_1-t_2)}\\
      & \quad +\frac{t_1 t_2 e^{\{y_2\} t_2}}{(e^{t_2}-1) (-t_1+t_2)}
      +\frac{t_1 t_2 e^{(1-\{y_1-y_2\}) t_1+\{y_1\} t_2}}{(e^{t_1}-1) (e^{t_2}-1)}
      \\
      &\quad
      -\frac{t_1 t_2 e^{(1-\{2 y_1-y_2\}) t_1}}{(e^{t_1}-1) (t_1+t_2)}
      -\frac{t_1 t_2 e^{\{2 y_1-y_2\} t_2}}{(e^{t_2}-1) (t_1+t_2)}
      \\
      & =\sum_{k_1,k_2=1}^\infty \mathcal{P}_{\Delta_{l+}}(k_1,k_2,y_1,y_2;\Delta)\frac{t_1^{k_1}t_2^{k_2}}{k_1!k_2!}.
  \end{align*}
Set $(y_1,y_2)=(0,0)$ and $\mathbf{k}=(0,k_1,k_2,0)$. Then $\zeta_2(0,k_1,k_2,0;y_1,y_2;\Delta)=\zeta_2(k_1,k_2)$ for $\Delta=\Delta(C_2)$. Hence it follows from \eqref{Th-main} that
  \begin{align}
& (1+(-1)^{k_1})(1+(-1)^{k_2})    \zeta_2(k_1,k_2)+(1+(-1)^{k_2})(1+(-1)^{k_1})    \zeta_2(k_2,k_1) \label{exam-C2} \\
& =     (-1)^4\mathcal{P}_{\Delta_{l+}}(k_1,k_2,0,0;\Delta)\frac{(2\pi i)^{k_1+k_2}}{k_1!k_2!}  \notag   
  \end{align}
for $k_1,k_2\geq 2$. 

For example, we can compute 
$$\mathcal{P}_{\Delta_{l+}}(4,4,0,0;\Delta)=\frac{1}{6300}$$ 
from the above expansion. Hence we obtain
$$\zeta_2(4,4)=\frac{(-1)^4}{8}\frac{1}{6300} \frac{(2\pi i)^8}{(4!)^2}=\frac{\pi^8}{113400}.$$
Similarly we can compute $\zeta_{2}(2k,2k)$ for 
$k\in \mathbb{N}$, though in this case we can also compute $\zeta_{2}(2k,2k)$ 
by using the well-known harmonic product formula for double zeta values 
\begin{equation}
\zeta(s)\zeta(t)=\zeta_{2}(s,t)+\zeta_{2}(t,s) +\zeta(s+t). \label{harm}
\end{equation}
In the next section, we introduce a slight generalization of Corollary \ref{Cor-Z} which gives evaluation formulas of $\zeta_{2}(k,l)$ for odd $k+l$ in terms of $\zeta(s)$ (see Proposition \ref{Pr-1}). 
\end{example}

\begin{remark}
In the general $C_r$ case, considering the expansion of 
$$(\mathfrak{D}_{\Delta_{s+}}F)({\bf t}_{\Delta_{l+}},{\bf 0};\Delta(C_r))$$
similarly, we can systematically compute $\zeta_{r}(2k,\ldots,2k)$. Moreover, 
considering the case $\nu\not={\bf 0}$ for $\zeta_r(\mathbf{s},\nu;\Delta(C_r))$, 
we can give character analogues of Corollary \ref{Cor-Z} for multiple $L$-values, 
which were first proved by Yamasaki \cite{Ya}. 
\end{remark}

\section{Some relations and parity results for double and triple zeta values}\label{sec-5}

In Theorem \ref{thm:main1}, we considered the sum over $W$ on the left-hand side of \eqref{Th-main}. 
Here, more generally, we consider the sum over a certain set of minimal coset
representatives on the left-hand side of \eqref{Th-main}. 
In this case, it is not easy to execute its computation directly. 
Hence we use a more technical method which was already introduced in \cite{KMT-CJ}. First we show the following result for double zeta values corresponding to a sub-root system of type $C_2$, where the number of the terms on the left-hand side is just the half of that on the left-hand side of \eqref{exam-C2}.

\begin{prop}\label{Pr-1}
For $p,q \in \mathbb{N}_{\geq 2}$, 
\begin{align*}
& \left( 1+(-1)^p\right)\zeta_{2}(p,q)+\left( 1+(-1)^q\right) \zeta_{2}(q,p) \\
& \ =2\sum_{j=0}^{[p/2]}\binom{p+q-2j-1}{q-1}\zeta(2j)\zeta(p+q-2j) \\
& \quad +2\sum_{j=0}^{[q/2]}\binom{p+q-2j-1}{p-1}\zeta(2j)\zeta(p+q-2j) -\zeta(p+q). 
\end{align*}
\end{prop}

\begin{proof}
The proof was essentially stated in \cite[Theorem 3.1]{KMT-CJ} which is a simpler form of a previous result for zeta-functions of type $A_2$ given by the third-named author \cite[Theorem 4.5]{Ts-Cam}. In fact, setting $(k,l,s)=(p,q,0)$ in \cite[Theorem 3.1]{KMT-CJ}, we have
\begin{align*}
& \zeta(p)\zeta(q)+(-1)^p\zeta_{2}(p,q)+(-1)^q \zeta_{2}(q,p) \\
& \ =2\sum_{j=0}^{[p/2]}\binom{p+q-2j-1}{q-1}\zeta(2j)\zeta(p+q-2j) \\
& \quad +2\sum_{j=0}^{[q/2]}\binom{p+q-2j-1}{p-1}\zeta(2j)\zeta(p+q-2j). 
\end{align*}
Combining this and \eqref{harm}, 
we have the assertion. 
\end{proof}

In particular when $p$ and $q$ are of different parity, we see that $\zeta_{2}(p,q)\in \mathbb{Q}[\{\zeta(j+1)\,|\,j\in \mathbb{N}\}]$ which was first proved by Euler. For example, we have
$$\zeta_2(2,3)=3\zeta(2)\zeta(3)-\frac{11}{2}\zeta(5).$$

Next we consider triple zeta values. From the viewpoint of the root system of $C_3$ type, we have the following theorem. Note that, unlike the case of double zeta values, this result seems not to be led from the result on the case of type $A_3$ (cf. \cite[Theorems 5.9 and 5.10]{MT1}).

\begin{thrm} \label{T-5-1}\ For $a,b,c\in \mathbb{N}_{\geq 2}$,
\begin{align*}
  &(1+(-1)^a)\zeta_3(a,b,c)+(1+(-1)^b)\{ \zeta_3(b,a,c)+\zeta_3(b,c,a)\}+(-1)^b(1+(-1)^c)\zeta_3(c,b,a)\\
& =2\bigg\{ \sum_{\xi=0}^{[a/2]}\zeta(2\xi)\sum_{\w=0}^{a-2\xi}\binom{\w+b-1}{\w}\binom{a+c-2\xi-\w-1}{c-1}\zeta_2(b+\w,a+c-2\xi-\w)\\
&  +\sum_{\xi=0}^{[b/2]}\zeta(2\xi)\sum_{\w=0}^{a-1}\binom{\w+b-2\xi}{\w}\binom{a+c-\w-2}{c-1}\zeta_2(b-2\xi+\w+1,a+c-1-\w)\\
&  +(-1)^b\sum_{\xi=0}^{[c/2]}\zeta(2\xi)\sum_{\w=0}^{c-2\xi}\binom{\w+b-1}{\w}\binom{a+c-2\xi-\w-1}{a-1}\zeta_2(b+\w,a+c-2\xi-\w)\\
&  +(-1)^b\sum_{\xi=0}^{[b/2]}\zeta(2\xi)\sum_{\w=0}^{c-1}\binom{\w+b-2\xi}{\w}\binom{a+c-\w-2}{a-1}\zeta_2(b-2\xi+\w+1,a+c-1-\w)\bigg\}\\
&  -\zeta_2(a+b,c)-(1+(-1)^b)\zeta_2(b,a+c)-(-1)^b\zeta_2(b+c,a).
\end{align*}
\end{thrm}

The proof of this theorem will be given in Section \ref{sec-proof2}.

This theorem especially implies the following result which was proved by Borwein and Girgensohn (see \cite{BG}). 

\begin{cor} \label{C-5-3} \ Let 
$$\mathfrak{X}=\mathbb{Q}\left[ \left\{ \zeta(j+1),\zeta_2(k,l+1)\right\}_{j,k,l\in \mathbb{N}}\right],$$
namely the $\mathbb{Q}$-algebra generated by Riemann zeta values and double zeta values at positive integers except singularities. 
Suppose $a,b,c\in \mathbb{N}_{\geq 2}$ satisfy that $a+b+c$ is even. Then $\zeta_3(a,b,c)\in \mathfrak{X}$.
\end{cor}

\begin{proof}
We recall the harmonic product formula
\begin{align}
& \zeta_3(a,b,c)+\zeta_3(b,a,c)+\zeta_3(b,c,a)  =\zeta(a)\zeta_2(b,c)-\zeta_2(b,c+a)-\zeta_2(a+b,c) \label{harmonic}
\end{align}
for $a,b,c\in \mathbb{N}_{\geq 2}$ (see \cite{Ka}). 

Let $a,b,c \in \mathbb{N}_{\geq 2}$ satisfying that $a+b+c$ is even. First we assume that $a,b,c$ are all even. Then, combining Theorem \ref{T-5-1} and \eqref{harmonic}, we see that $\zeta_3(c,b,a)\in \mathfrak{X}$. 

Next we assume that $a$ is even and $b,c$ are odd. Then, by Theorem \ref{T-5-1}, we see that $\zeta_3(a,b,c)\in \mathfrak{X}$. 

As for other cases, we can similarly obtain the assertions by using Theorem \ref{T-5-1} and \eqref{harmonic}. Thus we complete the proof.
\end{proof}

\begin{remark}
The following property of the multiple zeta value is sometimes called the parity result:

\it The multiple zeta value $\zeta_r(k_1,k_2,\ldots,k_r)$ of depth $r$ can be expressed as a rational linear combination of products of MZVs of lower depth than $r$, when its depth $r$ and its weight $\sum_{j=1}^{n}k_j$ are of different parity. 

\rm
The fact in case of depth 2 was proved by Euler, and that of depth $3$ was proved by Borwein and Girgensohn (see \cite{BG}). Further they conjectured the above assertion in the case of an arbitrary depth.   This conjecture was proved by the third-named author \cite{TsActa04} and by Ihara, Kaneko and Zagier \cite{IKZ} independently.
It should be stressed that our Corollary \ref{C-5-3} gives an explicit expression of the parity result for the triple zeta value under the condition $a,b,c\in \mathbb{N}_{\geq 2}$. 

Therefore it seems important to generalize Theorem \ref{T-5-1} in order to give an explicit expression of the parity result of an arbitrary depth.
\end{remark}

\begin{example}
Putting $(a,b,c)=(2,2,4)$ in Theorem \ref{T-5-1}, we have
\begin{align*}
& 2\zeta_3(2,2,4)+2\{\zeta_3(2,2,4)+\zeta_3(2,4,2)\}+2\zeta_3(4,2,2)\\
& =2\zeta(4)\zeta_2(2,2)+\zeta(2)\{8\zeta_2(4,2)+ 12\zeta_2(3,3)+16\zeta_2(2,4)+16\zeta_2(1,5)\} \\
& \quad -16\zeta_2(6,2) - 20\zeta_2(5,3)-25\zeta_2(4,4)-24\zeta_2(3,5)-17\zeta_2(2,6).
\end{align*}
Therefore, using \eqref{harmonic}, we obtain
\begin{align*}
\zeta_3(4,2,2)& =\zeta(4)\zeta_2(2,2)+\zeta(2)\{4\zeta_2(4,2)+ 6\zeta_2(3,3)+7\zeta_2(2,4)+8\zeta_2(1,5)\}\\
& \quad -8\zeta_2(6,2) - 10\zeta_2(5,3)-\frac{23}{2}\zeta_2(4,4)-12\zeta_2(3,5)-\frac{15}{2}\zeta_2(2,6) \in \mathfrak{X}.
\end{align*}
Note that this formula can be proved by combining known results for MZVs given by the double shuffle relations and harmonic product formulas (see, for example, \cite[Section 5]{MP}).
\end{example}

\begin{remark} 
If we replace \eqref{5-1-0} (in Section \ref{sec-proof2}) by 
$$\sum_{l\in \mathbb{N}}\sum_{m\in \mathbb{Z}^*} (-1)^{l+m}x^l y^m e^{i(l+m)\theta}, $$
and argue along the same line as in the proof of Theorem \ref{T-5-1}, then we can obtain 
\begin{align*}
& \left( 1+(-1)^a\right)\left( 1+(-1)^c\right)\left\{ \zeta_3(a,b,c)+\zeta_3(a,c,b)+\zeta_3(c,a,b)\right\} \\
& \ +\left( 1+(-1)^b\right)\left( 1+(-1)^c\right)\left\{ \zeta_3(c,b,a)+\zeta_3(b,c,a)+\zeta_3(b,a,c)\right\}\\
& \qquad \in \mathbb{Q}[\{\zeta(j+1)\,|\,j\in \mathbb{N}\}]
\end{align*}
for $a,b,c\in \mathbb{N}_{\geq 2}$. 
In particular when $a,b,c$ are both even, we have \eqref{Zagier-F} for the triple zeta value which can be regarded as a kind of Witten's volume formula \eqref{Zagier-F2} (see Section \ref{sec-4}). Furthermore, when 
$a$ is odd and both $b$ and $c$ are even, then 
\begin{align*}
& {\zeta_3(c,b,a)+\zeta_3(b,c,a)+\zeta_3(b,a,c)}  \in \mathbb{Q}\left[ \left\{\zeta(j+1)\,|\,j\in \mathbb{N}\right\}\right]. 
\end{align*}
Note that this result can also be deduced by combining \eqref{harmonic} and Proposition \ref{Pr-1}. 
\end{remark}

\section{Multiple zeta values associated with the root system of type $B_r$}
\label{sec-6}

In this section we discuss the $B_r$-analogue of our theory developed in the preceding
two sections.

As for the root system of type $B_r$, namely for $\Delta=\Delta(B_r)$, 
we see that
\begin{align*}
  (\Delta_{s+})^\vee&
  =\{\alpha_r^\vee,\ 2\alpha_{r-1}^\vee+\alpha_r^\vee, 
    2\alpha_{r-2}^\vee+2\alpha_{r-1}^\vee+\alpha_r^\vee,
    \ldots, 2\alpha_1^\vee+\cdots+2\alpha_{r-1}^\vee+\alpha_r^\vee  \}.
\end{align*}
Therefore for $\mathbf{s}_s=(s_{\alpha})_{\alpha\in\Delta_{s+}}$ we have
\begin{align}
& \zeta_r({\bf s}_s,{\bf 0};\Delta_{s+}(B_r))=\sum_{m_1,\ldots,m_r=1}^\infty
\prod_{i=1}^{r}
\frac{1}{(2\sum_{j=r-i+1}^{r-1}m_j+m_r)^{{s_i}}}, \label{B2-zeta}
\end{align}
which is a partial sum of $\zeta_{r}({\bf s})$. For example, we have
\begin{align}
& \zeta_2({\bf s}_s,{\bf 0};\Delta_{s+}(B_2))=\sum_{l,m=1}^\infty \frac{1}{m^{s_1}(2l+m)^{s_2}},\label{6-1}\\
& \zeta_3({\bf s}_s,{\bf 0};\Delta_{s+}(B_3))=\sum_{l,m,n=1}^\infty \frac{1}{n^{s_1}(2m+n)^{s_2}(2l+2m+n)^{s_3}},\label{6-2}
\end{align}
where $s_j=s_{\alpha_j}$ corresponding to $\alpha_j\in \Delta_{s+}$. 

From the viewpoint of zeta-functions 
of root systems, values of \eqref{B2-zeta} at positive integers can be regarded as the 
objects dual to MZVs, in the sense that $B_r$ and $C_r$ are dual of each other. 
Hence we denote \eqref{B2-zeta} by $\zeta_r^\sharp(s_1,\ldots,s_r)$. 

Since $W(B_r)\simeq W(C_r)$, just like Corollary \ref{Cor-Cr-Sr}, 
from Theorem \ref{thm:main1} we can obtain 
the following result, which gives the ``refined restricted sum formulas'' for
$\zeta_r^{\sharp}(\mathbf{s})$.
\begin{cor}\label{Cor-Br-Sr}
Let $\Delta=\Delta(B_r)$. For $(2{\bf k})_s=(2k_{\alpha})_{\alpha\in\Delta_{s+}}=(2k_1,\ldots,2k_r)\in \left(2\mathbb{N}\right)^r$ and 
$\mathbf{y}=\nu\in P^\vee/Q^\vee$, 
\begin{equation}
  \sum_{\sigma\in\mathfrak{S}_r}
  \zeta_r (\sigma^{-1}(2\mathbf{k})_s,\nu;\Delta_{l+})
=\frac{(-1)^{r}}{2^r}\mathcal{P}_{\Delta_{s+}}
((2\mathbf{k})_s,\nu;\Delta)
  \prod_{j=1}^{r}\frac{(2\pi i)^{2k_j}}{(2k_j)!}  \in\mathbb{Q}\cdot \pi^{2\sum_{j=1}^{r}k_j}. 
\end{equation}
In particular when $\nu={\bf 0}$, 
\begin{equation} \label{EZ-Sr-1-2}
  \sum_{\sigma\in\mathfrak{S}_r}
  \zeta_r^\sharp(2k_{\sigma^{-1}(1)},\ldots,2k_{\sigma^{-1}(r)})
=\frac{(-1)^{r}}{2^r}\mathcal{P}_{\Delta_{s+}}
((2\mathbf{k})_s,{\bf 0};\Delta)
  \prod_{j=1}^{r}\frac{(2\pi i)^{2k_j}}{(2k_j)!}  \in\mathbb{Q}\cdot \pi^{2\sum_{j=1}^{r}k_j}. 
\end{equation}
\end{cor}
%
%
%
From Theorem \ref{thm:main2}, we obtain an analogue of Corollary \ref{Cor-Z}, which is a kind of Witten's volume formula and also a $B_r$-type analogue of \eqref{Zagier-F}. 

\begin{cor} \label{C-6-2}
Let $\Delta=\Delta(B_r)$. For $({\bf 2k})_s=(2k,\ldots,2k)$ with any $k\in \mathbb{N}$, 
\begin{align*}
& \zeta_r^\sharp(2k,\ldots,2k)=\frac{(-1)^{r}}{2^r r!}\mathcal{P}_{\Delta_{s+}}
((\mathbf{2k})_s,{\bf 0};\Delta)\prod_{j=1}^{r}\frac{(2\pi i)^{2k_j}}{(2k_j)!}
\in\mathbb{Q}\cdot \pi^{2kr}.
\end{align*}
\end{cor}

\begin{example}\label{B-EZ-Exam}
\begin{align*}
\zeta_2^\sharp(2,2)&=\sum_{m,n=1}^\infty \frac{1}{n^{2} (2m+n)^{2}}=\frac{1}{320}\pi^4,\\
\zeta_2^\sharp(4,4)&=\sum_{m,n=1}^\infty \frac{1}{n^{4} (2m+n)^{4}}=\frac{23}{14515200}\pi^8,\\
\zeta_2^\sharp(6,6)&=\sum_{m,n=1}^\infty \frac{1}{n^{6} (2m+n)^{6}}=\frac{1369}{871782912000}\pi^{12}.
\end{align*}
These formulas can be obtained by calculating the generating function of type $B_2$ 
similarly to the case of type $C_2$ in Example \ref{Exam-C2} (see Section \ref{sec-4}). Also we can obtain these formulas by Theorem \ref{T-B2-EZ} in the case $(p,q)=(2k,2k)$ for $k\in \mathbb{N}$. 
However, unlike the ordinary 
double zeta value, 
these cannot be easily deduced from \eqref{harm}. 

Similarly, calculating the generating function of type $B_3$, we have explicit examples of Corollary \ref{C-6-2}:
\begin{align*}
\zeta_3^\sharp(2,2,2)&=\sum_{l,m,n=1}^\infty \frac{1}{n^{2} (2m+n)^{2}(2l+2m+n)^2}=\frac{1}{40320}\pi^{6},\\
\zeta_3^\sharp(4,4,4)&=\sum_{l,m,n=1}^\infty \frac{1}{n^{4} (2m+n)^{4}(2l+2m+n)^4}=\frac{23}{697426329600}\pi^{12},\\
\zeta_3^\sharp(6,6,6)&=\sum_{l,m,n=1}^\infty \frac{1}{n^{6} (2m+n)^{6}(2l+2m+n)^6}=\frac{1997}{17030314057236480000}\pi^{18}.
\end{align*}
\end{example}

Also, similarly to Proposition \ref{Pr-1}, we can obtain the following result whose proof will be given in Section \ref{sec-proof2}.

\begin{thrm}\label{T-B2-EZ}
For $p,q\in \mathbb{N}_{\geq 2}$,
\begin{align}
& \ (1+(-1)^p)\zeta_2^\sharp (p,q) +(1+(-1)^q)\zeta_2^\sharp(q,p)\label{Pr-2-1}\\
& = 2 \sum_{j=0}^{[p/2]} \frac{1}{2^{p+q-2j}}\binom{p+q-1-2j}{q-1}\zeta(2j)\zeta(p+q-2j)\notag\\
& + 2\sum_{j=0}^{[q/2]} \frac{1}{2^{p+q-2j}}\binom{p+q-1-2j}{p-1}\zeta(2j)\zeta(p+q-2j) -\zeta(p+q). \notag
\end{align}
\end{thrm}

Theorem \ref{T-B2-EZ} in the case that $p$ and $q$ are of different parity implies the following.

\begin{cor} \label{parity-B2}
Let $p,q \in \mathbb{N}_{\geq 2}$. Suppose $p$ and $q$ are of different parity, then   
$$\zeta_2^\sharp(p,q)\in \mathbb{Q}\left[ \left\{\zeta(j+1)\,|\,j\in \mathbb{N}\right\}\right],$$
which is a parity result for $\zeta_2^\sharp$. 
\end{cor}

\begin{remark}
This parity result for $\zeta_2^\sharp(p,q)$ is important in a recent study of the dimension of the linear space spanned by double zeta values of level $2$ given by Kaneko and Tasaka (see \cite{Ka-Ta}).

For example, setting $(p,q)=(3,2)$ in \eqref{Pr-2-1}, we have
\begin{align*}
\zeta_2^\sharp(2,3)&=\sum_{m,n=1}^\infty \frac{1}{n^{2} (2m+n)^{3}}=-\frac{21}{32}\zeta(5)+\frac{3}{8}\zeta(2)\zeta(3).
\end{align*}
It should be noted that this property can be given by combining the known facts for double zeta values and for their alternating series 
$$\varphi_2(s_1,s_2)=\sum_{m,n=1}^\infty \frac{(-1)^m}{n^{s_1}(m+n)^{s_2}}.$$
Actually we see that 
$$\zeta_2^\sharp(s_1,s_2)=\frac{1}{2}\left\{\zeta_2(s_1,s_2)+\varphi_2(s_1,s_2)\right\}.$$
When $p$ and $q$ are of different parity ($p,q \in \mathbb{N}$ and $q\geq 2$), Euler proved that 
$$\zeta_2(p,q)\in \mathbb{Q}\left[ \left\{\zeta(j+1)\,|\,j\in \mathbb{N}\right\}\right],$$
and Borwein et al. proved that 
$$\varphi_2(p,q)\in \mathbb{Q}\left[ \left\{\zeta(j+1)\,|\,j\in \mathbb{N}\right\}\right]$$
(see \cite{BBG}), from which Corollary \ref{parity-B2} follows.   
However \eqref{Pr-2-1} gives more explicit information on the parity result for $\zeta_2^\sharp(p,q)$. 
\end{remark}

Furthermore we can obtain the following result which can be regarded as an analogue of Theorem \ref{T-5-1} for type $B_3$. This can be proved similarly to Theorem \ref{T-5-1}, hence we omit its proof here.

\begin{thrm}\label{T-5-2}
For $a,b,c\in \mathbb{N}_{\geq 2}$,
\begin{align*}
  &(1+(-1)^a)\zeta_3^\sharp(a,b,c)+(1+(-1)^b)\{ \zeta_3^\sharp(b,a,c)+\zeta_3^\sharp(b,c,a)\}+(-1)^b(1+(-1)^c)\zeta_3^\sharp(c,b,a)\\
& =2^{1-a-b-c}\bigg\{ \sum_{\xi=0}^{[a/2]}2^\xi \zeta(2\xi)\sum_{\w=0}^{a-2\xi}\binom{\w+b-1}{\w}\binom{a+c-2\xi-\w-1}{c-1}\zeta_2(b+\w,a+c-2\xi-\w)\\
&  +\sum_{\xi=0}^{[b/2]}2^\xi\zeta(2\xi)\sum_{\w=0}^{a-1}\binom{\w+b-2\xi}{\w}\binom{a+c-\w-2}{c-1}\zeta_2(b-2\xi+\w+1,a+c-1-\w)\\
&  +(-1)^b\sum_{\xi=0}^{[c/2]}2^\xi\zeta(2\xi)\sum_{\w=0}^{c-2\xi}\binom{\w+b-1}{\w}\binom{a+c-2\xi-\w-1}{a-1}\zeta_2(b+\w,a+c-2\xi-\w)\\
&  +(-1)^b\sum_{\xi=0}^{[b/2]}2^\xi\zeta(2\xi)\sum_{\w=0}^{c-1}\binom{\w+b-2\xi}{\w}\binom{a+c-\w-2}{a-1}\zeta_2(b-2\xi+\w+1,a+c-1-\w)\bigg\}\\
&  -\zeta_2^\sharp(a+b,c)-(1+(-1)^b)\zeta_2^\sharp(b,a+c)-(-1)^b\zeta_2^\sharp(b+c,a).
\end{align*}
\end{thrm}

\begin{remark}
In \cite{KMT-Lie}, we study zeta-functions of weight lattices of 
semisimple compact connected Lie groups.   We can prove analogues of Theorem \ref{thm:main1} 
for those zeta-functions by a method similar to the above. 
We will give the details in a forthcoming paper.
\end{remark}


\section{Certain restricted sum formulas for $\zeta_r({\bf s})$ and for $\zeta_r^\sharp({\bf s})$} \label{sumf}

In this section, we give certain restricted sum formulas for $\zeta_r({\bf s})$ and for $\zeta_r^\sharp({\bf s})$ of an arbitrary depth $r$ which essentially include known results. 

As we stated in Section \ref{sec-1}, 
Gangl, Kaneko and Zagier \cite{GKZ} obtained the restricted sum formulas \eqref{F-GKZ} for double zeta values. Recently 
Nakamura \cite{Na-Sh} gave certain analogues of \eqref{F-GKZ}. 

More recently, Shen and Cai \cite{Shen-Cai} gave the following restricted sum formulas for triple and fourth zeta values:
\begin{align}
& \sum_{a_1,a_2,a_3 \in \mathbb{N}\atop a_1+a_2+a_3=N}
\zeta_3(2a_1,2a_2,2a_3)= \frac{5}{8}\zeta(2N) - \frac{1}{4}\zeta(2)\zeta(2N - 2)\in \mathbb{Q}\cdot \pi^{2N}\quad (N\in \mathbb{Z}_{\geq 3}), \label{sumf-triple}\\
& \sum_{a_1,a_2,a_3,a_4 \in \mathbb{N}\atop a_1+a_2+a_3+a_4=N}
\zeta_4(2a_1,2a_2,2a_3,2a_4)\label{sumf-fourth}\\
& \quad = \frac{35}{64}\zeta(2N) - \frac{5}{16}\zeta(2)\zeta(2N - 2)\in \mathbb{Q}\cdot \pi^{2N}(N\in \mathbb{Z}_{\geq 4}). \notag
\end{align}
Also Machide \cite{Mach} gave certain restricted sum formulas for triple zeta values. 

Now recall our Corollaries \ref{Cor-Cr-Sr} and \ref{Cor-Br-Sr}.
In the above restricted sum formulas, the summations are taken over all tuples
$(a_1,\ldots,a_r)$ satisfying $a_1+\cdots+a_r=N$.   On the other hand, the summations
in the formulas of Corollaries \ref{Cor-Cr-Sr} and \ref{Cor-Br-Sr} are running over
much smaller range, that is, just all the permutations of one fixed
$(a_1,\ldots,a_r)$ with $a_1+\cdots+a_r=N$.    Therefore our Corollaries give
subdivisions, or refinements, of known restricted sum formulas.

Summing our formulas for all tuples $(a_1,\ldots,a_r)$ satisfying $a_1+\cdots+a_r=N$,
we can obtain the $r$-ple generalization of \eqref{F-GKZ}, \eqref{sumf-triple} and 
\eqref{sumf-fourth}.
Moreover we can show the following further generalization, which gives a new type of
restricted sum formulas.

For $d\in \mathbb{N}$ and $N\in \mathbb{N}$, let
$$I_{r}(d,N)=\left\{ (2da_1,\ldots,2da_r)\in (2d\mathbb{N})^r\,|\,a_1+\cdots+a_r=N\right\}.$$
Denote by $P_r$ the set of all partitions of $r$, namely
$$P_r=\bigcup_{\nu=1}^{r}\{(j_1,\cdots,j_\nu)\in \mathbb{N}^\nu\,|\,j_1+\cdots+j_\nu=r\}.$$
For $J=(j_1,\cdots,j_\nu)\in P_r$, we set
$$\mathcal{A}_r(d,N,J)=\left\{ ((2dh_1)^{[j_1]},\ldots,(2dh_\nu)^{[j_\nu]})\in I_{r}(d,N)\,|\, h_1<\cdots<h_\nu\right\},$$
where $(2h)^{[j]}=(2h,\ldots,2h)\in (2\mathbb{N})^j$. 
Then we have the following restricted sum formulas of depth $r$.

\begin{thrm} \label{sumf-EZ-Cr}
For $d\in \mathbb{N}$ and $N\in \mathbb{N}$ with $N\geq r$, 
\begin{align}
& \sum_{a_1,\ldots,a_r \in \mathbb{N} \atop a_1+\cdots+a_r=N}\zeta_r(2da_1,\ldots,2da_r)\label{R-sumf}\\
& =\frac{(-1)^{r}}{2^r}\sum_{J=(j_1,\cdots,j_\nu)\in P_r}\frac{1}{j_1!\cdots j_\nu !}\notag\\
& \qquad \times \sum_{(2d\mathbf{k})_l \in \mathcal{A}_r(d,N,J)}\mathcal{P}_{\Delta_{l+}}
((2d\mathbf{k})_l,{\bf 0};\Delta(C_r))  \prod_{\rho=1}^{r}\frac{(2\pi i)^{2dk_\rho}}{(2k_\rho)!} \in\mathbb{Q}\cdot \pi^{2dN}.\notag 
\end{align}
\end{thrm}

\begin{remark}
In the case $d=1$ and $r=2,3,4$, we essentially obtain \eqref{F-GKZ}, \eqref{sumf-triple}, \eqref{sumf-fourth}. Also, in the case $N=r$, we obtain \eqref{Zagier-F2} stated in Corollary \ref{Cor-Z}. More generally, 
in the case $d=1$ and $r\geq 2$, Muneta \cite{Mu} already conjectured an explicit expression of the left-hand side of \eqref{R-sumf} in terms of $\{\zeta(2k)\,|\,k\in \mathbb{N}\}$.
\end{remark}

\begin{proof}[Proof of Theorem \ref{sumf-EZ-Cr}] 
Let $(2da_1,\ldots,2da_r)\in I_{r}(d,N)$.
Denote a set of different elements in $\{a_1,\ldots,a_r\}$ by
$\{h_1,\ldots,h_\nu\}$,  and put
$j_\mu=\sharp \{ a_m\,|\, a_m=h_\mu\}$ $(1\leq \mu\leq \nu)$.
We may assume $h_1<\cdots<h_{\nu}$. 
We can easily see that there exist $\sigma\in \mathfrak{S}_r$ and 
$((2dh_1)^{[j_1]},\ldots,(2dh_\nu)^{[j_\nu]}) \in \mathcal{A}_r(d,N,J)$ with 
$J=(j_1,\cdots,j_\nu)\in P_r$ such that 
$$(2da_1,\ldots,2da_r)=((2dh_1)^{[j_1]},\ldots,(2dh_\nu)^{[j_\nu]})^\sigma,$$
where we use the notation
$$(k_1,\ldots,k_r)^\sigma=(k_{\sigma(1)},\ldots,k_{\sigma(r)}).$$
On the other hand, the set 
$\{\left( (2dh_1)^{[j_1]},\ldots,(2dh_\nu)^{[j_\nu]}\right)^\tau\,|\,\tau\in \mathfrak{S}_r\}$ contains $j_1!\cdots j_\nu!$-copies of each element. 
In fact, if we denote by $\frak{S}(1,...,j_1)$ the set of all permutations among
$\{1,...,j_1\}$, then 
$$\mathfrak{X}(J):=\frak{S}(1,\dots,j_1) \times \frak{S}(j_1+1,\ldots,j_1+j_2)
             \times\cdots\times\frak{S}(\sum_{\rho=1}^{\nu-1}j_\rho+1, \ldots,\sum_{\rho=1}^{\nu}j_\rho) \subset \frak{S}_r
$$
forms the stabilizer subgroup of $((2dh_1)^{[j_1]},\ldots,(2dh_\nu)^{[j_\nu]})$, 
and hence $\sharp \mathfrak{X}(J)=j_1!\cdots j_\nu!$. 
Therefore, using Corollary \ref{Cor-Cr-Sr}, we have 
\begin{align*}
& \sum_{a_1,\ldots,a_r \in \mathbb{N} \atop a_1+\cdots+a_r=N}
\zeta_r(2da_1,\ldots,2da_r)=\sum_{(2da_1,\ldots,2da_r)\in I_{r}(d,N)}\zeta_r(2da_1,\ldots,2da_r) \\
& =\sum_{J=(j_1,\cdots,j_\nu)\in P_r}\frac{1}{j_1!\cdots j_\nu !}\sum_{(2dk_1,\ldots,2dk_r) \atop \in \mathcal{A}_r(d,N,J)}\sum_{\sigma\in \mathfrak{S}_r}\zeta_r(2dk_{\sigma(1)},\ldots,2dk_{\sigma(r)})\\
& =\frac{(-1)^{r}}{2^r}\sum_{J=(j_1,\cdots,j_\nu)\in P_r}\frac{1}{j_1!\cdots j_\nu !}\sum_{(2d\mathbf{k})_l \in \mathcal{A}_r(d,N,J)}\mathcal{P}_{\Delta_{l+}(C_r)}
((2d\mathbf{k})_l,{\bf 0};\Delta)  \prod_{\rho=1}^{r}\frac{(2\pi i)^{2dk_\rho}}{(2dk_\rho)!}.
\end{align*}
This completes the proof.
\end{proof}

Similarly, using Corollary \ref{Cor-Br-Sr}, we obtain the following.

\begin{thrm} \label{sumf-EZ-Br}
For $d\in \mathbb{N}$ and $N\in \mathbb{N}$ with $N\geq r$, 
\begin{align*}
& \sum_{a_1,\ldots,a_r \in \mathbb{N} \atop a_1+\cdots+a_r=N}\zeta_r^\sharp (2da_1,\ldots,2da_r)\\
& =\frac{(-1)^{r}}{2^r}\sum_{J=(j_1,\cdots,j_\nu)\in P_r}\frac{1}{j_1!\cdots j_\nu !}\notag\\
& \quad \times \sum_{(2d\mathbf{k})_s \in \mathcal{A}_r(d,N,J)}\mathcal{P}_{\Delta_{s+}}
((2d\mathbf{k})_s,{\bf 0};\Delta(B_r))  \prod_{\rho=1}^{r}\frac{(2\pi i)^{2dk_\rho}}{(2k_\rho)!} \in\mathbb{Q}\cdot \pi^{2dN}.
\end{align*}
\end{thrm}

\section{Analytically closed subclass}\label{sec-acs}

In this section we observe our theory from the analytic point of view.

First consider the case of type $C_r$.   In Section \ref{sec-4} we have shown that
the zeta-functions corresponding to the sub-root system of type $C_r$ consisting
of all long roots are exactly the family of Euler-Zagier sums.   On the other hand,
it is known that the Euler-Zagier $r$-fold sum can be expressed as an integral
involving the Euler-Zagier $(r-1)$-fold sum in the integrand.   In fact, it holds 
that
\begin{align}\label{acs-1}
\zeta_r(s_1,\ldots,s_r)=\frac{1}{2\pi i}\int_{(\kappa)}\frac{\Gamma(s_r+z)\Gamma(-z)}
{\Gamma(s_r)}\zeta_{r-1}(s_1,\ldots,s_{r-2},s_{r-1}+s_r+z)\zeta(-z)dz
\end{align}
for $r\geq 2$, where $-\Re s_r<\kappa<-1$ and the path of integral is the vertical line
from $\kappa-i\infty$ to $\kappa+i\infty$ (see \cite[Section 12]{Mat-NMJ}, 
\cite[Section 3]{Mat-JNT}).    This formula is proved by applying the classical
Mellin-Barnes integral formula (\eqref{acs-2} below), 
so we may call \eqref{acs-1} the Mellin-Barnes
integral expression of $\zeta_r(s_1,\ldots,s_r)$.

Formula \eqref{acs-1} implies that the family of Euler-Zagier sums is closed
under the Mellin-Barnes integral operation.   (Note that the Riemann zeta-function,
also appearing in the integrand, is the Euler-Zagier sum with $r=1$.)
When some family of zeta-functions is closed in this sense, we call the family
{\it analytically closed}.
The aim of this section is to prove that the subclasses of type $B_r$
and of type $A_r$ discussed in our theory are both analytically closed.

\begin{prop}\label{prop-acs-1}
The family of zeta-functions $\zeta_r({\bf s}_s,{\bf 0};\Delta_{s+}(B_r))$ defined by
\eqref{B2-zeta} is analytically closed.
\end{prop}

\begin{proof}
Recall the Mellin-Barnes formula
\begin{align}\label{acs-2}
(1+\lambda)^{-s}=\frac{1}{2\pi i}\int_{(\kappa)}\frac{\Gamma(s+z)\Gamma(-z)}
{\Gamma(s)}\lambda^z dz,
\end{align}
where $s,\lambda\in\mathbb{C}$ with $\Re s>0$, $\lambda\neq 0$,
$|\arg\lambda|<\pi$, $\kappa$ is real with $-\Re s<\kappa<0$.

Dividing the factor $(2(m_1+\cdots+m_{r-1})+m_r)^{-s_r}$ as
$$
(2(m_2+\cdots+m_{r-1})+m_r)^{-s_r}
\left(1+\frac{2m_1}{2(m_2+\cdots+m_{r-1})+m_r}\right)^{-s_r}
$$
and applying \eqref{acs-2} to the second factor with
$\lambda=2m_1/(2(m_2+\cdots+m_{r-1})+m_r)$, we obtain
\begin{align}
&\zeta_r((s_1,\ldots,s_r),{\bf 0};\Delta_{s+}(B_r))\label{acs-3}\\
&=\frac{1}{2\pi i}\int_{(\kappa)}\frac{\Gamma(s_r+z)\Gamma(-z)}{\Gamma(s_r)}
  \sum_{m_1,\ldots,m_r=1}^{\infty}\prod_{i=1}^{r-1}\frac{1}
  {(2\sum_{j=r-i+1}^{r-1}m_j+m_r)^{s_i}}\notag\\
&\qquad\times(2(m_2+\cdots+m_{r-1})+m_r)^{-s_r}
  \left(\frac{2m_1}{2(m_2+\cdots+m_{r-1})+m_r}\right)^z dz\notag\\
&=\frac{1}{2\pi i}\int_{(\kappa)}\frac{\Gamma(s_r+z)\Gamma(-z)}{\Gamma(s_r)}
  \sum_{m_1=1}^{\infty}(2m_1)^z\notag\\
&\;\times\sum_{m_2,\ldots,m_r=1}^{\infty}
  \prod_{i=1}^{r-2}\frac{1}{(2\sum_{j=r-i+1}^{r-1}m_j+m_r)^{s_i}}
  (2(m_2+\cdots+m_{r-1})+m_r)^{-s_{r-1}-s_r-z}dz\notag\\
&=\frac{1}{2\pi i}\int_{(\kappa)}\frac{\Gamma(s_r+z)\Gamma(-z)}{\Gamma(s_r)}
   2^z \zeta(z)
   \zeta_{r-1}((s_1,\ldots,s_{r-2},s_{r-1}+s_r+z),{\bf 0};\Delta_{s+}(B_{r-1}))dz.\notag
\end{align} 
This implies the assertion.
\end{proof}

Next we consider the subclass of type $A_r$ which we studied in \cite{KMT-MZ},
and prove that it is also analytically closed.   This part may be regarded as
a supplement of \cite{KMT-MZ}.

The explicit form of the zeta-function of the root system of type $A_r$ is given by
\begin{align}\label{acs-4}
\zeta_r(\mathbf{s},\mathbf{0};\Delta(A_r))=\sum_{m_1,\ldots,m_r=1}^{\infty}\prod_{h=1}^r
\prod_{j=h}^r\left(\sum_{k=h}^{r+h-j}m_k\right)^{-s_{hj}}
\end{align}
(where $\mathbf{s}=(s_{hj})_{h,j}$; see \cite[formula (13)]{KMT-MZ}).   
Let $a,b\in\mathbb{N}$, $c\in\mathbb{N}_0$ with $a+b+c=r$.
The main result in \cite{KMT-MZ} asserts
that the shuffle product procedure can be completely described by the partial
fraction decomposition of zeta-functions \eqref{acs-4} at special values 
$\mathbf{s}=\mathbf{d}=(d_{hj})_{h,j}$, where $d_{hj}$ for
\begin{align}\label{acs-5}
\left\{
\begin{array}{lll}h=1,\;1\leq j\leq c\\
h=1,\;b+c+1\leq j\leq a+b+c\\
h=a+1,\;a+c+1\leq j\leq a+b+c
\end{array}
\right.
\end{align}
are all positive integers, and all other $d_{hj}$ are equal to 0.
Let $\Delta_+^{(a,b,c)}=\Delta_+^{(a,b,c)}(A_r)$ be the set of all positive roots 
corresponding to $s_{hj}$
with $(h,j)$ in the list \eqref{acs-5}.   Then this is a root set, and the
above special values can be interpreted as special values of zeta-functions of
$\Delta_+^{(a,b,c)}$.
 
\begin{thrm}\label{Th-A}
The family of zeta-functions $\zeta_r(\mathbf{s}^{(a,b,c)},\mathbf{0};
\Delta_+^{(a,b,c)}(A_r))$ is analytically closed, where
$\mathbf{s}^{(a,b,c)}=(s_{hj})_{h,j}$ with $(h,j)$ in the list \eqref{acs-5}.
\end{thrm}

\begin{proof}
We prove that zeta-functions $\zeta_{r+1}$ belonging to the above family can be
expressed as a Mellin-Barnes integral, or multiple integrals, involving
$\zeta_r$ also belonging to the above family.
Let $a,b\in\mathbb{N}$, $c\in\mathbb{N}_0$ with $a+b+c=r$.
We show that all of the zeta-functions $\zeta_{r+1}$ associated with
(i) $\Delta_+^{(a+1,b,c)}$, (ii) $\Delta_+^{(a,b+1,c)}$, (iii) $\Delta_+^{(a,b,c+1)}$
have integral expressions involving the zeta-function of $\Delta_+^{(a,b,c)}$.

From \eqref{acs-4} we see that
\begin{align}
\zeta_r(\mathbf{s}^{(a,b,c)},\mathbf{0};                    
\Delta_+^{(a,b,c)}(A_r))
&=\sum_{m_1,\ldots,m_{a+b+c}=1}^{\infty}
      \prod_{j=1}^c(m_1+m_2+\cdots+m_{a+b+c+1-j})^{-s_{1j}}\label{acs-6}\\
&\times\prod_{j=b+c+1}^{a+b+c}(m_1+m_2+\cdots+m_{a+b+c+1-j})^{-s_{1j}}\notag\\
&\times\prod_{j=a+c+1}^{a+b+c}(m_{a+1}+m_{a+2}+\cdots+m_{2a+b+c+1-j})
      ^{-s_{a+1,j}},
\notag
\end{align}
which is, by renaming the variables, 
\begin{align}
=&\sum_{m_1,\ldots,m_{a+b+c}=1}^{\infty}(m_1+\cdots+m_{a+b+1})^{-s_{11}}\cdots
        (m_1+\cdots+m_{a+b+c})^{-s_{1c}}\label{acs-7}\\
&\times m_1^{-s_{21}}(m_1+m_2)^{-s_{22}}\cdots(m_1+\cdots+m_a)^{-s_{2a}}\notag\\
&\times m_{a+1}^{-s_{31}}(m_{a+1}+m_{a+2})^{-s_{32}}\cdots
        (m_{a+1}+\cdots+m_{a+b})^{-s_{3b}}.\notag
\end{align}

Now we consider the above three cases (i), (ii) and (iii) separately.

The simplest case is (iii).   When we replace $c$ by $c+1$ in \eqref{acs-7},
the differences are that the summation is now with respect to
$m_1,\ldots,m_{a+b+c+1}$, and a new factor
$(m_1+\cdots+m_{a+b+c+1})^{-s_{1,c+1}}$ appears.
Dividing this factor as
\begin{align*}
\lefteqn{(m_1+\cdots+m_{a+b+c+1})^{-s_{1,c+1}}}\\
&=(m_1+\cdots+m_{a+b+c})^{-s_{1,c+1}}
\left(1+\frac{m_{a+b+c+1}}{m_1+\cdots+m_{a+b+c}}\right)^{-s_{1,c+1}}
\end{align*}
and apply \eqref{acs-2} as in the argument of \eqref{acs-3},
we find that the sum with respect to $m_{a+b+c+1}$ is separated, which produces
a Riemann zeta factor, and hence the zeta-function of $\Delta_+^{(a,b,c+1)}$
can be expressed as an integral of Mellin-Barnes type, involving gamma factors,
a Riemann zeta factor, and the zeta-function of $\Delta_+^{(a,b,c)}$.

Next consider the case (ii).   When we replace $b$ by $b+1$, \eqref{acs-7} is
changed to
\begin{align}                                                                           
=&\sum_{m_1,\ldots,m_{a+b+c+1}=1}^{\infty}(m_1+\cdots+m_{a+b+2})^{-s_{11}}\cdots
        (m_1+\cdots+m_{a+b+c+1})^{-s_{1c}}\label{acs-8}\\                                      
&\times m_1^{-s_{21}}(m_1+m_2)^{-s_{22}}\cdots(m_1+\cdots+m_a)^{-s_{2a}}\notag\\        
&\times m_{a+1}^{-s_{31}}(m_{a+1}+m_{a+2})^{-s_{32}}\cdots                              
        (m_{a+1}+\cdots+m_{a+b})^{-s_{3b}}\notag\\
&\times(m_{a+1}+\cdots+m_{a+b+1})^{-s_{3,b+1}}. \notag                            
\end{align}
The last factor is
\begin{align}
&=(m_{a+1}+\cdots+m_{a+b})^{-s_{3,b+1}}\left(1+\frac{m_{a+b+1}}
    {m_{a+1}+\cdots+m_{a+b}}\right)^{-s_{3,b+1}}\label{acs-9}\\
&=(m_{a+1}+\cdots+m_{a+b})^{-s_{3,b+1}}\notag\\
&\qquad\times\frac{1}{2\pi i}\int_{(\kappa)}\frac{\Gamma(s_{3,b+1}+z)
   \Gamma(-z)}{\Gamma(s_{3,b+1})}\left(\frac{m_{a+b+1}}{m_{a+1}+\cdots+m_{a+b}}
   \right)^z dz.\notag
\end{align}
The factors $(m_1+\cdots+m_{a+b+n})^{-s_{1,n-1}}$ ($2\leq n\leq c+1$) also
include the term $m_{a+b+1}$.   We divide these factors as
\begin{align*}
\lefteqn{(m_1+\cdots+m_{a+b}+m_{a+b+2}+\cdots+m_{a+b+n})^{-s_{1,n-1}}}\\
&\times\left(1+\frac{m_{a+b+1}}{m_1+\cdots+m_{a+b}+m_{a+b+2}+\cdots+m_{a+b+n}}
\right)^{-s_{1,n-1}}
\end{align*}
and apply \eqref{acs-2} to obtain
\begin{align}
\lefteqn{(m_1+\cdots+m_{a+b+n})^{-s_{1,n-1}}}\label{acs-10}
\\
&=(m_1+\cdots+m_{a+b}+m_{a+b+2}+\cdots+m_{a+b+n})^{-s_{1,n-1}}\notag\\
&\times\frac{1}{2\pi i}\int_{(\kappa_n)}\frac{\Gamma(s_{1,n-1}+z_n)                     
   \Gamma(-z_n)}{\Gamma(s_{1,n-1})}\left(\frac{m_{a+b+1}}
   {m_1+\cdots+m_{a+b}+m_{a+b+2}+\cdots+m_{a+b+n}}\right)^{z_n}dz_n \notag
\end{align}
for $2\leq n\leq c+1$.   Substituting \eqref{acs-9} and \eqref{acs-10} into
\eqref{acs-8}, we find that the sum with respect to $m_{a+b+1}$ is separated 
and gives a Riemann zeta factor
$\zeta(-z_2-\cdots-z_{c+1}-z)$.
Since the remaining sum produces the zeta-function of 
$\Delta_+^{(a,b,c)}$, we obtain that the zeta-function of 
$\Delta_+^{(a,b+1,c)}$ can be expressed as a $(c+1)$-ple integral of
Mellin-Barnes type involving $\zeta(-z_2-\cdots-z_{c+1}-z)$ and the 
zeta-function of $\Delta_+^{(a,b,c)}$.

The case (i) is similar; we omit the details, only noting that in this case the 
variable to be separated is $m_{a+1}$.   The proof of Theorem \ref{Th-A} is now
complete.
\end{proof}

\section{Proof of fundamental formulas}\label{sec-proof1}

In this section we prove fundamental formulas stated in Section \ref{sec-3}.

\begin{lem}
For $B\subset \Delta_+$ and $\mathbf{V}\in\mathscr{V}$, we have
  \begin{equation}
    {\rm L.h.}[\mathbf{V}\cap B]
    =\{v\in V~|~\text{$\langle v,\mu^{\mathbf{V}}_\beta\rangle=0$ for all $\beta\in\mathbf{V}\setminus B$}\}.
  \end{equation}
\end{lem}
\begin{proof}
Let $v$ be an element of the right-hand side.
We write $v=\sum_{\beta\in\mathbf{V}}c_\beta \beta$
and have $c_\beta=0$ for all $\beta\in\mathbf{V}\setminus B$
and hence
\begin{equation}
  v=\sum_{\beta\in\mathbf{V}\cap B}c_\beta \beta\in{\rm L.h.}[\mathbf{V}\cap B].
\end{equation}
The converse is shown similarly.
\end{proof}

\begin{proof}[Proof of Theorem \ref{thm:main0}]
For $\mathbf{t}=(t_{\alpha})_{\alpha\in\Delta_+}\in\mathbf{T}$, $\mathbf{y}\in V$, 
$\mathbf{V}\in\mathscr{V}$, $B\subset \Delta_+$ and $q\in Q^\vee/L(\mathbf{V}^\vee)$, let
\begin{equation}
  \begin{split}
      F(\mathbf{t},\mathbf{y};\mathbf{V},B,q)&=
  (-1)^{\abs{B\setminus\mathbf{V}}}
  \Bigl(
  \prod_{\gamma\in\Delta_+\setminus(\mathbf{V}\cup B)}
  \frac{t_\gamma}
  {t_\gamma-\sum_{\beta\in\mathbf{V}\setminus B}t_\beta\langle\gamma^\vee,\mu^{\mathbf{V}}_\beta\rangle}
  \Bigr)
\\
&\qquad\times
  \Bigl(
  \prod_{\beta\in\mathbf{V}\setminus B}\frac{t_\beta\exp
    (t_\beta\{\mathbf{y}+q\}_{\mathbf{V},\beta})}{e^{t_\beta}-1}
  \Bigr),
\end{split}
\end{equation}
so that
\begin{equation}\label{termwise}
  F(\mathbf{t},\mathbf{y};\Delta)=
      \sum_{\mathbf{V}\in\mathscr{V}}
      \frac{1}{\abs{Q^\vee/L(\mathbf{V}^\vee)}}
      \sum_{q\in Q^\vee/L(\mathbf{V}^\vee)}
      F(\mathbf{t},\mathbf{y};\mathbf{V},\emptyset,q).
\end{equation}

Assume
$\mathbf{y}\in V\setminus\mathfrak{H}_{\mathscr{R}}$,
and let
\begin{equation}
  F_j=F(\mathbf{t},\mathbf{y};\mathbf{V},A_j,q).
\end{equation}
We calculate $\mathfrak{D}_{\nu_{j+1}} F_j$.
First, since $\mathbf{y}\notin\mathfrak{H}_{\mathscr{R}}$, noting Remark \ref{tsuika}
we find that
\begin{equation}
  \partial_{\nu_{j+1}^\vee} F_j
  =
  \Bigl(\sum_{\beta\in\mathbf{V}\setminus A_j}t_\beta\langle\nu_{j+1}^\vee,\mu^{\mathbf{V}}_\beta\rangle\Bigr)
  F_j.
\end{equation}
Consider the case $\nu_{j+1}\in\mathbf{V}$.
Then $\langle\nu_{j+1}^\vee,\mu^{\mathbf{V}}_\beta\rangle=\delta_{\nu_{j+1},\beta}$ and
\begin{equation}
  \sum_{\beta\in\mathbf{V}\setminus A_j}t_\beta\langle\nu_{j+1}^\vee,
\mu^{\mathbf{V}}_\beta\rangle=t_{j+1},
\end{equation}
where we write $t_{\nu_{j+1}}=t_{j+1}$ for brevity.
Hence we have $\partial_{\nu_{j+1}^\vee} F_j =t_{j+1} F_j$.
Therefore we obtain
\begin{equation}
  \begin{split}
    \mathfrak{D}_{\nu_{j+1}} F_j
    &=
    (-1)^{\abs{A_j\setminus\mathbf{V}}}
    \Bigl(
    \prod_{\gamma\in\Delta_+\setminus(\mathbf{V}\cup A_j)}
    \frac{t_\gamma}
    {t_\gamma-\sum_{\beta\in\mathbf{V}\setminus(A_j\cup\{\nu_{j+1}\})}t_\beta\langle\gamma^\vee,\mu^{\mathbf{V}}_\beta\rangle}
    \Bigr)
    \\
    &
    \qquad \times
    \Bigl(
    \prod_{\beta\in\mathbf{V}\setminus(A_j\cup\{\nu_{j+1}\})}\frac{t_\beta\exp
      (t_\beta\{\mathbf{y}+q\}_{\mathbf{V},\beta})}{e^{t_\beta}-1}
    \Bigr)
  \end{split}
\end{equation}
which is equal to $F_{j+1}$ because 
$\Delta_+\setminus(\mathbf{V}\cup (A_j\cup\{\nu_{j+1}\}))=\Delta_+
\setminus(\mathbf{V}\cup A_j)$ and
$\abs{(A_j\cup\{\nu_{j+1}\})\setminus\mathbf{V}}=\abs{A_j\setminus\mathbf{V}}$.

Next consider the case $\nu_{j+1}\notin\mathbf{V}$. 
If $\langle\nu_{j+1}^\vee,\mu^{\mathbf{V}}_\beta\rangle=0$ for all $\beta\in\mathbf{V}\setminus A_j$, then
\begin{equation}
    \partial_{\nu_{j+1}^\vee} F_j
  =
  \Bigl(\sum_{\beta\in\mathbf{V}\setminus A_j}t_\beta\langle\nu_{j+1}^\vee,\mu^{\mathbf{V}}_\beta\rangle\Bigr)
  F_j
=0
\end{equation}
and hence $\mathfrak{D}_{\nu_{j+1}} F_j=0$.
Otherwise, since
\begin{equation}
  \frac{\partial}{\partial t_{j+1}}\biggr\rvert_{t_{j+1}=0}
  \Bigl(\frac{t_{j+1}}
  {t_{j+1}-\sum_{\beta\in\mathbf{V}\setminus A_j}t_\beta\langle\nu_{j+1}^\vee,\mu^{\mathbf{V}}_\beta\rangle}
  \Bigr)
  =
  -\frac{1}{\sum_{\beta\in\mathbf{V}\setminus A_j}t_\beta\langle\nu_{j+1}^\vee,\mu^{\mathbf{V}}_\beta\rangle}
\end{equation}
we have
\begin{equation}
  \begin{split}
    \mathfrak{D}_\nu F_j&=
    (-1)^{\abs{A_j\setminus\mathbf{V}}+1}
    \Bigl(
    \prod_{\gamma\in\Delta_+\setminus(\mathbf{V}\cup A_j\cup\{\nu_{j+1}\})}
    \frac{t_\gamma}
    {t_\gamma-\sum_{\beta\in\mathbf{V}\setminus A_j}t_\beta\langle\gamma^\vee,\mu^{\mathbf{V}}_\beta\rangle}
    \Bigr)
    \\
    &\qquad\times
    \Bigl(
    \prod_{\beta\in\mathbf{V}\setminus A_j}\frac{t_\beta\exp
      (t_\beta\{\mathbf{y}+q\}_{\mathbf{V},\beta})}{e^{t_\beta}-1}
    \Bigr).
  \end{split}
\end{equation}
By noting
$\mathbf{V}\setminus (A_j\cup\{\nu_{j+1}\})= 
\mathbf{V}\setminus A_j$ and
$\abs{(A_j\cup \{\nu_{j+1}\})\setminus \mathbf{V}}=   
\abs{A_j\setminus \mathbf{V}}+1$
we find that the right-hand side is equal to $F_{j+1}$.

We see that
the condition
$\langle\nu_{j+1},\mu^{\mathbf{V}}_\beta\rangle=0$ for all $\beta\in\mathbf{V}\setminus A_j$
 is equivalent to 
the condition
$\nu_{j+1}\in{\rm L.h.}[\mathbf{V}\cap A_j]$.
Therefore the above results can be summarized as
\begin{equation}
  \mathfrak{D}_{\nu_{j+1}} F_j
  =
  \begin{cases}
    0\qquad &(\nu_{j+1}\in{\rm L.h.}[\mathbf{V}\cap A_j]),\\
    F_{j+1}\qquad &(\nu_{j+1}\notin{\rm L.h.}[\mathbf{V}\cap A_j]).
  \end{cases}
\end{equation}
Hence
\begin{equation}
\label{eq:DAF0} 
  \mathfrak{D}_A F_0
  =
  \begin{cases}
    0\qquad &(\mathbf{V}\notin\mathscr{V}_A),\\
    F_N\qquad &(\mathbf{V}\in\mathscr{V}_A).
  \end{cases}
\end{equation}
Similarly to the above calculations, we see that $\mathfrak{D}_{A,2} F_0$ gives the 
same result as \eqref{eq:DAF0}.
Thus, since $F_0=F(\mathbf{t},\mathbf{y};\mathbf{V},\emptyset,q)$, from
\eqref{termwise} we obtain \eqref{eq:main0}.

The continuity
follows from the limit
\begin{equation}
   \lim_{c\to0+}\{\mathbf{y}+q+c\phi\}_{\mathbf{V},\beta}
=\{\mathbf{y}+q\}_{\mathbf{V},\beta}
\end{equation}
(see the last part of the proof of \cite[Theorem 4.1]{KM5}.)
Finally, since $F(\mathbf{t},\mathbf{y};\Delta)$ is holomorphic with respect to $\mathbf{t}$ around the origin,
so is $\bigl(\mathfrak{D}_A F\bigr) (\mathbf{t}_{\Delta^*},\mathbf{y};\Delta)$
with respect to $\mathbf{t}_{\Delta^*}$.
The proof of Theorem \ref{thm:main0} is thus complete.
\end{proof}


\begin{proof}[Proof of Theorem \ref{thm:main1}]
  First assume  $\mathbf{y}\in V\setminus\mathfrak{H}_{\mathscr{R}}$.
  Let
  $\mathbf{k}'=(k'_\alpha)_{\alpha\in\Delta_+}$ with
  $k'_\alpha=k_\alpha$ ($\alpha\in \Delta^*$),
  $k'_\alpha=2$ ($\alpha\in \Delta_+\setminus\Delta^*=A$).
  Then by Proposition \ref{prop:ZP}, we have
  \begin{equation}
    \begin{split}
      S(\mathbf{k}',\mathbf{y};\Delta)
      &=
      \sum_{\lambda\in P\setminus H_{\Delta^\vee}}
      e^{2\pi i\langle \mathbf{y},\lambda\rangle}
      \prod_{\alpha\in\Delta_+}
      \frac{1}{\langle\alpha^\vee,\lambda\rangle^{k'_\alpha}}
      \\
      &=
      \sum_{w\in W}
      \Bigl(\prod_{\alpha\in\Delta_+\cap w\Delta_-}
      (-1)^{k'_{\alpha}}\Bigr)
      \zeta_r(w^{-1}\mathbf{k}',w^{-1}\mathbf{y};\Delta)
      \\
      &=(-1)^{\abs{\Delta_+}}  
      \mathcal{P}(\mathbf{k}',\mathbf{y};\Delta)
      \biggl(\prod_{\alpha\in\Delta_+}
      \frac{(2\pi i)^{k'_\alpha}}{k'_\alpha!}\biggr).
    \end{split}
  \end{equation}
Applying $\prod_{\alpha\in A}\partial_{\alpha^\vee}^2$ to the above.
From the first line we observe that each $\partial_{\alpha^\vee}^2$ produces
the factor $(2\pi i\langle\alpha^{\vee},\lambda\rangle)^2$.
Hence the factor
$\zeta_r(w^{-1}\mathbf{k}',w^{-1}\mathbf{y};\Delta)$ on the second line
is transformed into
$(2\pi i)^{2|A|}\zeta_r(w^{-1}\mathbf{k},w^{-1}\mathbf{y};\Delta)$.
Therefore we have
  \begin{multline}
    \label{eq:key1} 
    (2\pi i)^{2|A|}
    \sum_{w\in W}
    \Bigl(\prod_{\alpha\in\Delta_+\cap w\Delta_-}
    (-1)^{k_{\alpha}}\Bigr)
    \zeta_r(w^{-1}\mathbf{k},w^{-1}\mathbf{y};\Delta)
    \\
    =(-1)^{\abs{\Delta_+}} 
    \Bigl(\prod_{\alpha\in A}\partial_{\alpha^\vee}^2 \Bigr)
    \mathcal{P}(\mathbf{k}',\mathbf{y};\Delta)
    \biggl(\prod_{\alpha\in\Delta_+}
    \frac{(2\pi i)^{k'_\alpha}}{k'_\alpha!}\biggr).
  \end{multline}
Since
\begin{align*}
\biggl(\prod_{\alpha\in\Delta_+}                                                      
    \frac{(2\pi i)^{k'_\alpha}}{k'_\alpha!}\biggr)=
\biggl(\prod_{\alpha\in\Delta^*}                                                      
    \frac{(2\pi i)^{k'_\alpha}}{k'_\alpha!}\biggr)                                     
    \biggl(\prod_{\alpha\in A}      
    \frac{(2\pi i)^{2}}{2!}\biggr),
\end{align*}
we have
  \begin{equation}                                                                  \label{eq:key1bis} 
       \begin{split}
    &\sum_{w\in W}                                                                       
    \Bigl(\prod_{\alpha\in\Delta_+\cap w\Delta_-}                                       
    (-1)^{k_{\alpha}}\Bigr)                                                             
    \zeta_r(w^{-1}\mathbf{k},w^{-1}\mathbf{y};\Delta)                                   
    \\                                                                                  
    &=(-1)^{\abs{\Delta_+}}                                                              
    \Bigl(\prod_{\alpha\in A}\frac{1}{2}\partial_{\alpha^\vee}^2 \Bigr)
    \mathcal{P}(\mathbf{k}',\mathbf{y};\Delta)                                          
    \biggl(\prod_{\alpha\in\Delta^*}                                                    
    \frac{(2\pi i)^{k'_\alpha}}{k'_\alpha!}\biggr).                                      
    \end{split}
  \end{equation}
From \eqref{def-F} it follows that
  \begin{equation}
\label{eq:key2}
      \Bigl(\prod_{\alpha\in A}\frac{1}{2}\frac{\partial^2 }{\partial t_\alpha^2 }\biggr\rvert_{t_\alpha=0}\partial_{\alpha^\vee}^2 \Bigr)
      F(\mathbf{t},\mathbf{y};\Delta)
      \\
      =
      \sum_{\substack{\mathbf{m}=(m_\alpha)_{\alpha\in\Delta_+}\\m_\alpha\in
\mathbb{N}_{0}(\alpha\in\Delta^*)\\m_\alpha=2(\alpha\in A)}}
      \Bigl(\prod_{\alpha\in A}\frac{1}{2}\partial_{\alpha^\vee}^2 \Bigr)
      \mathcal{P}(\mathbf{m},\mathbf{y};\Delta)
      \prod_{\alpha\in\Delta^*} \frac{t_\alpha^{m_\alpha}}{m_\alpha!}.
    \end{equation}
    By Theorem \ref{thm:main0}, we see that
the left-hand side of \eqref{eq:key2} is equal to
  \begin{equation}
    \label{eq:key3}
    F_{\Delta^*}(\mathbf{t}_{\Delta^*},\mathbf{y};\Delta)
=
    \sum_{\mathbf{m}_{\Delta^*}\in \mathbb{N}_{0}^{\abs{\Delta^*}}}\mathcal{P}_{\Delta^*}(\mathbf{m}_{\Delta^*},\mathbf{y};\Delta)
    \prod_{\alpha\in\Delta^*}
    \frac{t_{\alpha}^{m_\alpha}}{m_\alpha!}.
  \end{equation}
Comparing \eqref{eq:key2} with \eqref{eq:key3} we find that
\begin{align*}
\Bigl(\prod_{\alpha\in A}\frac{1}{2}\partial_{\alpha^\vee}^2 \Bigr)                   
    \mathcal{P}(\mathbf{k}',\mathbf{y};\Delta)
=\mathcal{P}_{\Delta^*}(\mathbf{k}_{\Delta^*},\mathbf{y};\Delta).
\end{align*}
Therefore \eqref{eq:key1bis} implies the desired result when
$\mathbf{y}\in V\setminus\mathfrak{H}_{\mathscr{R}}$.
By the continuity with respect to $\mathbf{y}$,
the result is also valid in the case when
$\mathbf{y}\in\mathfrak{H}_{\mathscr{R}}$.
\end{proof}

\begin{remark}
  It is possible to prove Theorem \ref{thm:main1}
  by use of $\mathfrak{D}_A$ instead of $\mathfrak{D}_{A,2}$.
  In this method, we need to consider the case $k_\alpha=1$ for some $\alpha\in A$ 
  and such an argument is indeed valid.  (See \cite[Remark 3.2]{KM3}.)
\end{remark}

\section{Proofs of Theorems \ref{T-5-1} and \ref{T-B2-EZ}}\label{sec-proof2}

In this final section we prove Theorems \ref{T-5-1} and \ref{T-B2-EZ}.
The basic principle of the proofs of these theorems is similar to that of the
argument developed in \cite[Section 7]{KMT-CJ}.   
We first state the following lemma. 

\begin{lem} \label{L-5-2} \ For an 
arbitrary function $f\,:\, \mathbb{N}_{0} \to \mathbb{C}$ and $d\in \mathbb{N}$, we have
\begin{align}                                                                          
&\sum_{k=0}^{d}\phi(d-k)\ee_{d-k}\sum_{\nu=0}^{k}f(k-\nu)\frac{(i\pi)^{\nu}}{\nu!} 
=-\frac{i\pi}{2}f(d-1)+\sum_{\xi=0}^{[d/2]} \zeta(2\xi)f(d-2\xi), \label{MNOT}              
\end{align}
where we denote the integer part of $x\in \mathbb{R}$ by $[x]$, 
$\ee_j=(1+(-1)^j)/2$ $(j\in \mathbb{Z})$
and $\phi(s)=\sum_{m\geq 1}(-1)^m m^{-s}=\left(2^{1-s}-1\right)\zeta(s)$.
\end{lem}

\begin{proof}
This can be immediately obtained by combining (2.6) and (2.7) (with the choice
$g(x)=i\pi f(x-1)$$\;$)
in \cite[Lemma 2.1]{MNOT}.
\end{proof}

\begin{proof}[Proof of Theorem \ref{T-5-1}]
From \cite[(4.31) and (4.32)]{KMT-CJ}, we have
\begin{align}                                                                           
& \sum_{n\in \mathbb{Z}^*} \frac{(-1)^{n}e^{in\theta}}{n^a}-2\sum_{j=0}^{a}\ 
\phi(a-j)\ee_{a-j} \frac{(i\theta)^{j}}{j!}=0  \label{e-5-1}                          
\end{align}
for $a\geq 2$ and $\theta \in [-\pi,\pi]$, where 
$\mathbb{Z}^*=\mathbb{Z}\smallsetminus \{0\}$.   For $x,y \in \mathbb{R}$ with 
$|x|<1$
 and $|y|<1$, multiply the above by
\begin{equation}                                                                       
\sum_{l,m\in \mathbb{N}} (-1)^{l+m}x^l y^m e^{i(l+m)\theta}. \label{5-1-0}              
\end{equation}
Separating the terms corresponding to $l+m+n=0$, we obtain
\begin{align*}                                                                         
& \sum_{l,m\in \mathbb{N}}\sum_{n\in \mathbb{Z}^*\atop l+m+n\not=0} 
\frac{(-1)^{l+m+n}x^l y^m e^{i(l+m+n)\theta}}{n^a}\\                                  
& \ -2\sum_{j=0}^{a}\ \phi(a-j)\ee_{a-j}\sum_{l,m\in \mathbb{N}}(-1)^{l+m}x^l y^m  
e^{i(l+m)\theta} \frac{(i\theta)^{j}}{j!}\\                                           
& \ =-(-1)^a\sum_{l,m\in \mathbb{N}} \frac{x^l y^m}{(l+m)^a}                            
\end{align*}
for $\theta \in [-\pi,\pi]$.
The right-hand side of the above is constant with respect to $\theta$.
Therefore we can apply \cite[Lemma 6.2]{KMT-CJ} with 
$h=1$, $a_1=a$, $d=c\geq 2$, 
$$C(N)=\sum_{l,m\in\mathbb{N},n\in\mathbb{Z}^*\atop l+m+n=N}\frac{x^l y^m}{n^a},$$
\begin{align*}
D(N;r;1)=
   \begin{cases}
   \sum_{l,m\in\mathbb{N}\atop l+m=N}x^l y^m & (N\geq 2,r=0),\\
   0 & ({\rm otherwise})
   \end{cases}
\end{align*}
in the notation of \cite{KMT-CJ}.    The result is
\begin{align*}                                                                          
& \sum_{l,m\in \mathbb{N}}\sum_{n\in \mathbb{Z}^*\atop {l+m+n\not=0}} 
\frac{(-1)^{l+m+n}x^l y^m e^{i(l+m+n)\theta}}{n^a(l+m+n)^c} \\   
& \ -2\sum_{j=0}^{a}\ \phi(a-j)\ee_{a-j}\sum_{\xi=0}^{j} \binom{j-\xi+c-1}{j-\xi}
(-1)^{j-\xi}\sum_{l,m\in \mathbb{N}}\frac{(-1)^{l+m}x^l y^m e^{i(l+m)\theta}}
{(l+m)^{c+j-\xi}} \frac{(i\theta)^{\xi}}{\xi!}\\                                         
& \ +2\sum_{j=0}^{c}\ \phi(c-j)\ee_{c-j}\sum_{\xi=0}^{j} \binom{j-\xi+a-1}{a-1}
(-1)^{a-1}\sum_{l,m\in \mathbb{N}}\frac{x^l y^m }{(l+m)^{a+j-\xi}} \frac{(i\theta)^{\xi}}
{\xi!}=0.      
\end{align*}
Replace $x$ by $-xe^{-i\theta}$ and separate the term corresponding to $m+n=0$ in the 
first member on the left-hand side, and apply \cite[Lemma 6.2]{KMT-CJ} again with 
$d=b\geq 2$.   Then we can obtain
\begin{align}     
& \sum_{l,m\in \mathbb{N}}\sum_{n\in \mathbb{Z}^*\atop {m+n\not=0 \atop l+m+n\not=0}} 
\frac{(-1)^{m+n}x^l y^m e^{i(m+n)\theta}}{n^a(m+n)^b (l+m+n)^c} \label{triple-1}  \\                
& \ =2\sum_{j=0}^{a}\ \phi(a-j)\ee_{a-j}\sum_{\xi=0}^{j} \sum_{\w=0}^{j-\xi}
\binom{\w+b-1}{\w}(-1)^\w \binom{j-\xi-\w+c-1}{c-1}(-1)^{j-\xi-\w} \notag\\             
& \qquad \qquad \times \sum_{l,m\in \mathbb{N}}\frac{(-1)^{m}x^l y^m e^{im\theta}}
{m^{b+\w}(l+m)^{c+j-\xi-\w}} \frac{(i\theta)^{\xi}}{\xi!} \notag\\   
& \ -2\sum_{j=0}^{b}\ \phi(b-j)\ee_{b-j}\sum_{\xi=0}^{j} \sum_{\w=0}^{a-1}
\binom{\w+j-\xi}{\w}(-1)^\w \binom{a-1-\w+c-1}{c-1}(-1)^{a-1-\w} \notag\\   
& \qquad \qquad \times \sum_{l,m\in \mathbb{N}}\frac{x^l y^m}{m^{j-\xi+\w+1}
(l+m)^{a+c-1-\w}} \frac{(i\theta)^{\xi}}{\xi!} \notag\\      
& \ -2\sum_{j=0}^{c}\ \phi(c-j)\ee_{c-j}\sum_{\xi=0}^{j} \sum_{\w=0}^{j-\xi}
\binom{\w+b-1}{\w}(-1)^\w \binom{j-\xi-\w+a-1}{a-1}(-1)^{a-1} \notag\\               
& \qquad \qquad \times \sum_{l,m\in \mathbb{N}}\frac{(-1)^{l}x^l y^m e^{-il\theta}}
{(-l)^{b+\w}(l+m)^{a+j-\xi-\w}} \frac{(i\theta)^{\xi}}{\xi!} \notag\\  
& \ +2\sum_{j=0}^{b}\ \phi(b-j)\ee_{b-j}\sum_{\xi=0}^{j} \sum_{\w=0}^{c-1}
\binom{\w+j-\xi}{\w}(-1)^\w \binom{a-1-\w+c-1}{a-1}(-1)^{a-1} \notag\\                
& \qquad \qquad \times \sum_{l,m\in \mathbb{N}}\frac{x^l y^m}{(-l)^{j-\xi+\w+1}
(l+m)^{a+c-1-\w}} \frac{(i\theta)^{\xi}}{\xi!}.  \notag
\end{align}
Since $a,b,c \geq 2$, we can let $x,y \to 1$ on the both sides because of  
absolute convergence.  Then set $\theta=\pi$, and consider 
the left-hand side of the resulting formula first.
The contribution of the terms corresponding to $m+2n=0$ is obviously
$(-1)^a\zeta_2(a+b,c)$.
The contribution of the terms corresponding to $l+m+2n=0$ is (with rewriting
$-n$ by $n$) 
\begin{align*}
(-1)^a\sum_{m,n\in\mathbb{N}\atop m\neq n, m<2n}\frac{1}{n^{a+c}(m-n)^b},
\end{align*}
which is, by separating into two parts according to $n<m<2n$ and $0<m<n$, 
equal to $(-1)^a(1+(-1)^b)\zeta_2(b,a+c)$.   We can also see that the 
contribution of the terms corresponding to $l+2m+2n=0$ is
\begin{align*}
(-1)^a \sum_{m,n\in\mathbb{N}\atop n>m}\frac{1}{n^a(m-n)^b(n-m)^c}
    =(-1)^{a+b}\zeta_2(b+c,a).
\end{align*}
The remaining part of the left-hand side is
\begin{align*}                                                                          
& \sum_{l,m\in \mathbb{N}}\sum_{n\in \mathbb{Z}^*\atop {m+n\not=0 \atop {m+2n\not=0 
\atop {l+m+n\not=0 \atop {l+m+2n\not=0 \atop l+2m+2n\not=0}}}}} \frac{1}
{n^a(m+n)^b (l+m+n)^c} \notag\\                                  
& =\zeta_3(a,b,c)+(-1)^a\sum_{l,m\in \mathbb{N}}
\sum_{n\in \mathbb{N}\atop {m\not=n \atop 
{m\not=2n \atop {l+m\not=n \atop {l+m\not=2n \atop l+2m\not=2n}}}}} \frac{1}
{n^a(m-n)^b (l+m-n)^c}.                                          
\end{align*}
On the above double sum,
replace $j=m-n$ and $k=n-m$ correspondingly to $m>n$ and $m<n$, respectively.
On the part corresponding to $m>n$, we further divide the sum into three parts
according to $l+j<n$, $j<n<l+j$, $n<j$ and find that the contribution of this part is
$$
(-1)^a\left\{\zeta_3(b,c,a)+\zeta_3(b,a,c)+\zeta_3(a,b,c)\right\}.
$$
Similarly we treat the part $m<n$.
Collecting the above results, we obtain that the left-hand side is
\begin{align*}  
  (-1)^a&\bigg\{(1+(-1)^a)\zeta_3(a,b,c)+(1+(-1)^b)\left( \zeta_3(b,a,c)+
\zeta_3(b,c,a)\right)\\  
& \qquad +(-1)^b(1+(-1)^c)\zeta_3(c,b,a)+\zeta_2(a+b,c)\\                               
& \qquad +(1+(-1)^b)\zeta_2(b,a+c)+(-1)^b\zeta_2(b+c,a)\bigg\}.                         
\end{align*}
On the other hand, applying Lemma \ref{L-5-2}, we can rewrite the right-hand side 
to
\begin{align*}                                                                          
& 2(-1)^a\bigg\{ \sum_{\xi=0}^{[a/2]}\zeta(2\xi)\sum_{\w=0}^{a-2\xi}\binom{\w+b-1}
{\w}\binom{a+c-2\xi-\w-1}{c-1}\zeta_2(b+\w,a+c-2\xi-\w)\\                             
& \ +\sum_{\xi=0}^{[b/2]}\zeta(2\xi)\sum_{\w=0}^{a-1}\binom{\w+b-2\xi}{\w}
\binom{a+c-\w-2}{c-1}\zeta_2(b-2\xi+\w+1,a+c-1-\w)\\                                  
& \ +(-1)^b\sum_{\xi=0}^{[c/2]}\zeta(2\xi)\sum_{\w=0}^{c-2\xi}\binom{\w+b-1}{\w}
\binom{a+c-2\xi-\w-1}{a-1}\zeta_2(b+\w,a+c-2\xi-\w)\\                                 
& \ +(-1)^b\sum_{\xi=0}^{[b/2]}\zeta(2\xi)\sum_{\w=0}^{c-1}\binom{\w+b-2\xi}{\w}
\binom{a+c-\w-2}{a-1}\zeta_2(b-2\xi+\w+1,a+c-1-\w)\bigg\}.                           
\end{align*}
This completes the proof of Theorem \ref{T-5-1}.
\end{proof}

Finally we give the proof of Theorem \ref{T-B2-EZ}. 

\begin{proof}[Proof of Theorem \ref{T-B2-EZ}] 
Let $p\in \mathbb{N}_{\geq 2}$ and $s\in \mathbb{R}_{>1}$. It follows from 
\cite[Equation (4.7)]{KMT-Pala} that 
\begin{equation*}
\begin{split}
& \sum_{l\in \mathbb{Z}^*, m\in\mathbb{N}\atop l+m\not=0} \frac{(-1)^{l+m}x^m e^{i(l+m)\theta}}{l^{p}m^{s}}-2\sum_{j=0}^{p}\ \phi(p-j)\varepsilon_{p-j}\left\{ \sum_{m=1}^\infty \frac{(-1)^{m}x^m e^{im\theta}}{m^s}\right\} \frac{(i\theta)^{j}}{j!}\\
& \ \ \ \ +(-1)^{p}\sum_{m=1}^\infty \frac{x^m}{m^{s+p}}=0 
\end{split}
\end{equation*}
for $\theta \in [-\pi,\pi]$ and $x\in \mathbb{C}$ with $|x|\leq 1$. Setting $x=-e^{i\theta}$ on the both sides and separating the term corresponding to $l+2m=0$ of the first term on the left-hand side, we have
\begin{align*}
& \sum_{l\in \mathbb{Z}^*,m\in\mathbb{N}\atop {l+m\not=0 \atop l+2m\not=0}} \frac{(-1)^{l} e^{i(l+2m)\theta}}{l^{p}m^{s}} -2\sum_{j=0}^{p}\ \phi(p-j)\varepsilon_{p-j}\left\{ \sum_{m=1}^\infty \frac{ e^{2im\theta}}{m^s}\right\} \frac{(i\theta)^{j}}{j!}\\
& \ \ \ \ +(-1)^{p}\sum_{m=1}^\infty \frac{(-1)^me^{im\theta}}{m^{s+p}}=-\sum_{m=1}^\infty \frac{1}{(-2m)^p m^s}. 
\end{align*}
By \cite[Lemma 6.2]{KMT-CJ} with $d=q\geq 2$, we obtain
\begin{align}
& \sum_{l\in \mathbb{Z}^*,m\in\mathbb{N}\atop{l+m\not=0 \atop l+2m\not=0}} \frac{(-1)^{l} e^{i(l+2m)\theta}}{l^{p}m^{s}(l+2m)^q} \label{eq-9-2}\\
& \quad =2\sum_{j=0}^{p}\ \phi(p-j)\varepsilon_{p-j}\sum_{\xi=0}^{j}\binom{j-\xi+q-1}{j-\xi}\frac{(-1)^{j-\xi}}{2^{q+j-\xi}}\sum_{m=1}^{\infty}\frac{e^{2im\theta}}{m^{s+q+j-\xi}}\frac{(i\theta)^{\xi}}{\xi!}\notag\\
& \quad -2\sum_{j=0}^{q}\ \phi(q-j)\varepsilon_{q-j}\sum_{\xi=0}^{j}\binom{j-\xi+p-1}{j-\xi}\frac{(-1)^{p-1}}{2^{p+j-\xi}}\sum_{m=1}^{\infty}\frac{1}{m^{s+p+j-\xi}}\frac{(i\theta)^{\xi}}{\xi!}\notag\\
& \ \ \ \ -(-1)^{p}\sum_{m=1}^{\infty}\frac{(-1)^m e^{im\theta}}{m^{s+p+q}}. \notag
\end{align}
Let $\theta=\pi$ and using Lemma \ref{L-5-2}. Then the right-hand side of \eqref{eq-9-2} is equal to
\begin{align}
& 2(-1)^{p}\sum_{\xi=0}^{[p/2]}\ \frac{1}{2^{p+q-2\xi}}\binom{p+q-1-2\xi}{q-1}\zeta(2\xi)\zeta(s+p+q-2\xi)\label{eq-9-3} \\
& +2(-1)^{p}\sum_{\xi=0}^{[q/2]}\ \frac{1}{2^{p+q-2\xi}}\binom{p+q-1-2\xi}{p-1}\zeta(2\xi)\zeta(s+p+q-2\xi) \notag\\
& -(-1)^{p}\zeta(s+p+q). \notag
\end{align}
On the other hand, we can see that the left-hand side can be written in terms of
the zeta-function of $B_2$.   Recall that
\begin{align*}
\zeta_2(s_1,s_2,s_3,s_4;B_2)&=\zeta_2((s_1,s_2,s_3,s_4),{\bf 0};\Delta(B_2))\\
&=\sum_{m_1=1}^{\infty}\sum_{m_2=1}^{\infty}\frac{1}{m_1^{s_1}m_2^{s_2}
(m_1+m_2)^{s_3}(2m_1+m_2)^{s_4}}.
\end{align*}
The contribution of the terms with $l>0$ to the left-hand side is obviously
$\zeta_2(s,p,0,q;B_2)$.   As for the terms with $l<0$, we rewrite $-l$ by $l$, 
divide the sum into three parts according to the conditions $l<m$, $m<l<2m$
and $l>2m$, and evaluate each part in terms of the zeta-function of $B_2$.
The conclusion is that the left-hand side is
\begin{align}
& \zeta_2(s,p,0,q;B_2)+(-1)^p\zeta_2(0,p,s,q;B_2) +(-1)^p\zeta_2(0,q,s,p;B_2)\label{eq-9-4}\\
& \qquad +(-1)^{p+q}\zeta_2(s,q,0,p;B_2).\notag
\end{align}
We combine \eqref{eq-9-3} and \eqref{eq-9-4} and multiply by $(-1)^p$. Then we can set $s=0$ because \eqref{eq-9-3} and \eqref{eq-9-4} are absolutely convergent for $s>-1$. 
Noting $\zeta_2(0,p,0,q;B_2)=\zeta_2^\sharp(p,q)$, we complete the proof of Theorem \ref{T-B2-EZ}.
\end{proof}

\ 

\proof[Acknowledgements]
The authors would like to express their sincere gratitude to 
Professor Mike Hoffman for pointing out that symmetric sums for MZVs in \eqref{EZ-Sr-11} can be written in terms of products of Riemann's zeta values at even positive integers and giving related valuable comments (see Remark \ref{Rem-Hof}).

\


\begin{thebibliography}{999}

\bibitem{AK1}
T. Arakawa and M. Kaneko, {Notes on Multiple Zeta Values and Multiple $L$ Values}, Lecture Note, Rikkyo Univ., 2005 (in Japanese).

\bibitem{AK2}
T. Arakawa and M. Kaneko, \emph{Introduction to Multiple Zeta Values}, MI Lecture Note Vol.~23, Kyushu Univ., 2010, http://www.math.kyushu-u.ac.jp/\verb|~|mkaneko (in Japanese).

\bibitem{BBG}
D. Borwein, J. M. Borwein, and R. Girgensohn, Explicit evaluation of Euler sums, {Proc. Edinburgh Math. Soc.} {\bf 38} (1995),  277--294. 

\bibitem{BG}
J. M. Borwein and R. Girgensohn, Evaluation of triple Euler sums, {Electron. J. Combin.} {\bf 3} (1996), Research Paper 23, approx. 27 pp. 

\bibitem{BBB}
D. J. Broadhurst, J. M. Borwein, and D. M. Bradley, Evaluation of $k$-fold Euler/Zagier sums: a compendium of results for arbitrary $k$, Electron. J. Combin. {\bf 4(2)} (1997), Research Paper 5, approx. 21 pp.  

\bibitem{Bourbaki}
{N. Bourbaki,} \emph{Groupes et Alg{\`e}bres de Lie, Chapitres 4, 5 et 6}, Hermann, Paris, 1968.

\bibitem{GKZ}
H. Gangl, M. Kaneko and D. Zagier, Double zeta values and modular forms, in: {Automorphic Forms and Zeta Functions}, World Sci. Publ., Hackensack, NJ, 2006: pp.~71--106.

\bibitem{Hoff} 
M. E. Hoffman, Multiple harmonic series, {Pacific J. Math.} {\bf 152} (1992),  275--290.

\bibitem{Hum72}
{J. E. Humphreys,} \emph{Introduction to Lie Algebras and Representation Theory}, Graduate Texts in Mathematics, Vol.~9, Springer-Verlag, New York-Berlin, 1972.
\bibitem{Hum}
{J. E. Humphreys,} \emph{Reflection Groups and Coxeter Groups}, Cambridge University Press, Cambridge, 1990. 

\bibitem{IKZ}
{K. Ihara, M. Kaneko, and D. Zagier}, Derivation and double shuffle relations for multiple zeta values, {Compositio Math.} {\bf 142} (2006),  307--338.

\bibitem{Ka}
M. Kaneko, Multiple zeta values. {Sugaku Expositions} {\bf 18} (2005),  221--232 (translation of {Sugaku} {\bf 54} (2002),  404--415).

\bibitem{Ka-Ta}
M. Kaneko and K. Tasaka,
Double zeta values, double Eisenstein series, and modular forms of level $2$, preprint.

\bibitem{KMT}
{Y. Komori, K. Matsumoto and H. Tsumura,} 
Zeta-functions of root systems, in: {The Conference on $L$-functions}, L. Weng and M. Kaneko (eds.), World Scientific, 2007, pp.~115--140.

\bibitem{KMTpja}
{Y. Komori, K. Matsumoto and H. Tsumura,} 
Zeta and $L$-functions and Bernoulli polynomials of root systems, {Proc. Japan Acad.}, Series A, {\bf 84} (2008),  57--62.

\bibitem{KM2}
{Y. Komori, K. Matsumoto and H. Tsumura,} 
On Witten multiple zeta-functions associated with semisimple Lie algebras II, {J. Math. Soc. Japan} {\bf 62} (2010),  355--394. 

\bibitem{KM5}
{Y. Komori, K. Matsumoto and H. Tsumura,} 
On multiple Bernoulli polynomials and multiple $L$-functions of root systems, {Proc. London Math. Soc.} {\bf 100} (2010), 303--347.

\bibitem{KMT-Mem}
{Y. Komori, K. Matsumoto and H. Tsumura,} 
An introduction to the theory of zeta-functions of root systems, in: {Algebraic and Analytic Aspects of Zeta Functions and L-functions}, G. Bhowmik, K. Matsumoto and H. Tsumura (eds.), MSJ Memoirs, Vol. 21, Mathematical Society of Japan, 2010, pp.~115--140. 

\bibitem{KMT-CJ}
{Y. Komori, K. Matsumoto and H. Tsumura,} 
Functional relations for zeta-functions of root systems, in: {Number Theory: Dreaming in Dreams - Proceedings of the 5th China-Japan Seminar}, T. Aoki, S. Kanemitsu and J. -Y. Liu (eds.), World Sci. Publ., 2010, pp.~135--183.

\bibitem{KM3}
{Y. Komori, K. Matsumoto and H. Tsumura,} 
On Witten multiple zeta-functions associated with semisimple Lie algebras III,       in {Multiple Dirichlet Series, L-functions and 
      Automorphic Forms}, D. Bump, S. Friedberg and D. Goldfeld (eds), 
      Progress Math. 300, Birkh\"auser, 2012, pp.~223-286. 


\bibitem{KM4}
{Y. Komori, K. Matsumoto and H. Tsumura,} 
On Witten multiple zeta-functions associated with semisimple Lie algebras IV, {Glasgow Math. J.} {\bf 53} (2011),  185--206.

\bibitem{KMT-MZ}
{Y. Komori, K. Matsumoto and H. Tsumura,} 
Shuffle products for multiple zeta values and partial fraction decompositions of zeta-functions of root systems, {Math. Z.} {\bf 268} (2011),  993--1011.

\bibitem{KMT-PJA}
{Y. Komori, K. Matsumoto and H. Tsumura,} 
Multiple zeta values and zeta-functions of root systems, {Proc. Japan Acad.,} {\bf 87}, Ser. A (2011),  103--107.

\bibitem{KMT-Pala}
{Y. Komori, K. Matsumoto and H. Tsumura,} 
Functional relations for zeta-functions of weight lattices of Lie groups of type $A_3$, 
to appear in: {Anal. Probab. Methods Number Theory}, E. Manstavi{\v c}ius et al. (eds), TEV, 2012.

\bibitem{KMT-Lie}
{Y. Komori, K. Matsumoto and H. Tsumura,} 
Zeta-functions of weight lattices of compact connected semisimple Lie groups, Preprint, arXiv:1011.0323. 

\bibitem{Mach}
T. Machide, 
Extended double shuffle relations and the generating function 
of triple zeta values of any fixed weight, preprint, arXiv:1204.4085.

\bibitem{Mat-NMJ}
{K. Matsumoto,}
Asymptotic expansions of double zeta-functions of Barnes, of Shintani, and
Eisenstein series, 
{Nagoya Math. J.} {\bf 172} (2003),  59--102.

\bibitem{Mat-JNT}
{K. Matsumoto,}
The analytic continuation and the asymptotic behaviour of certain multiple 
zeta-functions I, 
{J. Number Theory} {\bf 101} (2003),  223--243.

\bibitem{MT1}
{K. Matsumoto and H. Tsumura,} 
On Witten multiple zeta-functions associated with semisimple Lie algebras I. 
{Ann. Inst. Fourier (Grenoble)} {\bf 56} (2006),  1457--1504.

\bibitem{MNOT}
K. Matsumoto, T. Nakamura, H. Ochiai and H. Tsumura, 
On value-relations, functional relations and singularities of Mordell-Tornheim
and related triple zeta-functions, {Acta Arith.} {\bf 132} (2008),  99--125.

\bibitem{MP}
H. N. Minh and M. Petitot, 
Lyndon words, polylogarithms, and the Riemann $\zeta$ function. {Discrete Math.} {\bf 217} (2000),  273-292. 

\bibitem{Mun}
S. Muneta, 
On some explicit evaluations of multiple zeta-star values, 
{J. Number Theory} {\bf 128} (2008),  2538--2548.

\bibitem{Mu}
S. Muneta, 
Refined sum formula of multiple zeta value, 
in: Proc. 3rd Fukuoka Number Theory Conference 
(Fukuoka, 2008), M. Kaneko, Y. Gon and Y. Kishi (eds), 2009, pp.~49--63 (in Japanese).

\bibitem{Na-Sh}
{T. Nakamura}, Restricted and weighted sum formulas for double zeta values of even weight, {{\v S}iauliai Math. Semin.} {\bf 4\,(12)} (2009),  151--155. 

\bibitem{Shen-Cai}
{Z. Shen and T. Cai}, 
Some identities for multiple zeta values, 
{J. Number Theory} {\bf 132} (2012),  314--323.

\bibitem{TsActa04}
{H. Tsumura}, 
Combinatorial relations for Euler-Zagier sums, {Acta Arith.} {\bf 111} (2004),  27--42.

\bibitem{Ts-Cam}
H. Tsumura, 
On functional relations between the Mordell-Tornheim double zeta functions and the Riemann zeta function, {Math. Proc. Cambridge Philos. Soc.} {\bf 142} (2007),  395--405. 

\bibitem{Wi}
{E. Witten}, On quantum gauge theories in two dimensions, {Commun.~Math.~Phys.} {\bf 141} (1991),  153--209.

\bibitem{Ya}
Y. Yamasaki, Evaluations of multiple Dirichlet $L$-values via symmetric functions, {J. Number Theory} {\bf 129} (2009),  2369--2386. 

\bibitem{Za}
{D. Zagier}, Values of zeta functions and their applications, in: {First European Congress of Mathematics, Vol.~II}, A. Joseph
   et al. (eds.), Progr.~Math.~{120}, Birkh{\"a}user, 1994,
   pp.~497--512.

\end{thebibliography}
\end{document}